\documentclass[preprint,12pt]{elsarticle}
\usepackage{psfrag}
\usepackage{a4wide}
\usepackage{rotating}
\usepackage{color}
\usepackage{amssymb}
\usepackage{amsmath}
\usepackage{amsthm}
\usepackage{verbatim}

\newcommand{\ab}[1]{\boldsymbol{#1}}

\newcommand{\N}{\mathbb N}

\newcommand{\R}{\mathbb R}

\newtheorem{thm}{Theorem}
\newtheorem{lem}[thm]{Lemma}
\newtheorem{prop}[thm]{Proposition}
\newtheorem{cor}[thm]{Corollary}

\theoremstyle{definition}
\newtheorem{ex}[thm]{Example}
\newtheorem{defn}[thm]{Definition}
\newtheorem{rem}[thm]{Remark}

\definecolor{dred}{rgb}{0.92,0,0}
\definecolor{dgreen}{rgb}{0,0.6,0}

\begin{document}

\begin{frontmatter}

\title{Dimension and basis construction for analysis-suitable \\ 
$G^{1}$ two-patch parameterizations}


\author[pavia1]{Mario Kapl\corref{cor}}
\ead{mario.kapl@unipv.it}
 
\author[pavia1,pavia2]{Giancarlo Sangalli}
\ead{giancarlo.sangalli@unipv.it}

\author[linz]{Thomas Takacs}
\ead{thomas.takacs@jku.at}

\address[pavia1]{Dipartimento di Matematica ``F.Casorati", Universit\`{a} degli Studi di Pavia, Italy}

\address[pavia2]{Istituto di Matematica Applicata e Tecnologie Informatiche ''E.Magenes`` (CNR), Italy}

\address[linz]{Institute of Applied Geometry, Johannes Kepler University Linz, Austria}

\cortext[cor]{Corresponding author}

\begin{abstract}
We study the dimension and construct a basis for $C^1$-smooth isogeometric function spaces over two-patch domains. In this context, an isogeometric function is a 
function defined on a B-spline domain, whose graph surface also has a B-spline representation. We consider constructions along one interface between two patches. We 
restrict ourselves to a special case of planar B-spline patches of bidegree~$(p,p)$ with $p \geq 3$, so-called analysis-suitable $G^1$~geometries, which are derived 
from a specific geometric continuity condition. This class of two-patch geometries is exactly the one which allows, under certain additional assumptions, $C^1$ 
isogeometric spaces with optimal approximation properties (cf. \cite{CoSaTa16}).

Such spaces are of interest when solving numerically fourth-order PDE problems, such as the biharmonic equation, using the isogeometric method. In particular, we 
analyze the dimension of the $C^1$-smooth isogeometric space and present an explicit representation for a basis of this space. Both the dimension of the space and 
the basis functions along the common interface depend on the considered two-patch parameterization. Such an explicit, geometry dependent basis construction is 
important for an efficient implementation of the isogeometric method. The stability of the constructed basis is numerically confirmed for an example configuration.
\end{abstract}


\end{frontmatter}

\section{Introduction} \label{sec:introduction}

The problems discussed in this paper are inspired by isogeometric analysis (IGA), which was developed in \cite{HuCoBa04}. 
The core idea of isogeometric analysis is to use the spline based representation of CAD models directly for the numerical analysis of partial differential equations 
(PDE). For a more detailed description of the isogeometric framework we refer to \cite{ANU:9260759,CottrellBook}.

One of the advantages of IGA is the possibility to have discretization spaces of high order smoothness. These spaces can then be used to directly solve high order PDE 
problems. There exist several fourth (and higher) order problems of practical relevance. For their application in IGA see \cite{BaDe15,TaDe14}, 
as well as \cite{ABBLRS-stream,da2012isogeometric,benson2011large,kiendl-bazilevs-hsu-wuechner-bletzinger-10,kiendl-bletzinger-linhard-09} for Kirchhoff-Love shells, 
\cite{gomez2008isogeometric} for the Cahn-Hilliard equation and \cite{gomez2010isogeometric} for the Navier-Stokes-Korteweg equation. 

The spline based representation of the physical domains allows for high order smoothness within one B-spline patch. However, most geometries of practical relevance 
cannot be represented directly with one patch but have to be parametrized using a multi-patch approach. It is not trivial to construct smooth function spaces over 
multi-patch domains. Many results for multi-patch domains can be derived from considerations on a single interface, hence in the following we will restrict ourselves to 
two-patch domains.

We are interested in $C^1$ isogeometric function spaces over two-patch domains. More precisely, we compute the dimension and construct a stable basis for a special 
class of, so called, analysis-suitable $G^1$ parameterizations. This class of geometries was first introduced in \cite{CoSaTa16}, where it was also shown that exactly 
the geometries of this class allow under certain assumptions $C^{1}$ isogeometric spaces with optimal 
approximation properties (cf. \cite{CoSaTa16}). Note that an
isogeometric function is $C^1$ if its graph surface is
$G^1$.
Hence, the analysis-suitability condition is a restriction of the more general geometric continuity condition. We refer to \cite{Peters2,Pe90} for the definition of 
geometric continuity. 

The existing literature about the construction of $C^{1}$-smooth isogeometric functions on two-patch (and multi-patch) domains can be roughly classified 
into two possible approaches. The first one employs $G^{1}$-surface constructions around extraordinary vertices to obtain a set of $C^{1}$-smooth functions, see e.g. 
\cite{Pe15-2,Peters2,NgKaPe15,NgPe16}. In contrast, the second approach studies the entire space of $C^{1}$-smooth isogeometric functions on any given two-patch 
(or multi-patch) parameterization and generate a basis of the corresponding $C^{1}$ isogeometric space. Examples are 
\cite{BeMa14,CoSaTa16,KaBuBeJu16,KaViJu15,Matskewich-PhD,mourrain2015geometrically}.

In this paper we follow the second approach, especially explored in \cite{BeMa14,KaBuBeJu16,KaViJu15,Matskewich-PhD}. We extend these results in two main directions. 
First, these existing constructions and investigations are limited to piecewise bilinear domains. Our approach encloses the much wider class of analysis-suitable 
$G^{1}$ parameterizations, which contains the class of piecewise bilinear domains. Second, our construction works for non-uniform splines of arbitrary bidegree $(p,p)$ 
with $p \geq 3$ and of arbitrary regularity $r$ with $1 \leq r < p-1$ within the two single patches. In contrast, the constructions \cite{BeMa14,Matskewich-PhD} are 
restricted to biquartic (for special cases) and to biquintic B\'ezier elements and the constructions \cite{KaBuBeJu16,KaViJu15} are restricted to bicubic and 
biquartic uniform splines of regularity $r=1$. 

Further differences to \cite{BeMa14,Matskewich-PhD} are that our approach allows the construction of nested $C^{1}$ isogeometric spaces and that our basis functions are 
explicitly given, whereas the basis functions in \cite{BeMa14,Matskewich-PhD} are implicitly defined by means of minimal determining sets for the B\'ezier coefficients.  
Similar to \cite{KaBuBeJu16,KaViJu15}, the explicitly given basis functions possess a very small local support and are well conditioned. Moreover, the spline coefficients 
of our basis functions can be simply obtained by means of blossoming or fitting. This could provide a simple implementation in existing IGA libraries. In addition, 
we present the study of the dimension of the resulting $C^{1}$ isogeometric spaces for all possible configurations of analysis-suitable $G^{1}$ two-patch 
parameterizations.

The remainder of the paper is organized as follows. In Section~\ref{sec:prel} we describe some basic definitions and notations which are used throughout the paper. 
Section~\ref{sec:C1-smooth} recalls the concept of $C^{1}$-smooth isogeometric spaces over analysis-suitable $G^{1}$ two-patch 
parameterizations. The dimension of these spaces, which depends on the considered two-patch parameterization, is analyzed in 
Section~\ref{sec:dimension}. Then we present in Section~\ref{sec:basis} an explicit construction of basis functions, and describe in Section~\ref{sec:matrices}  
their resulting spline coefficients by means of blossoming (and fitting). Finally, we conclude the paper in Section~\ref{sec:conclusion}.

\section{Preliminaries} \label{sec:prel}

Let $\omega$ be the interval $[0,1]$ or the unit square $[0,1]^2$. We denote by $\mathcal{S}(\mathcal{T}^{p,r}_{k},\omega)$ the (tensor-product) spline space of 
degree~$p$ (in each direction), which is defined on $\omega$ by choosing the open knot vector $\mathcal{T}^{p,r}_{k}=(t^{p,r}_{0},\ldots,t^{p,r}_{2p+1+k(p-r)})$ 
(in each direction) given by
\begin{equation*}  
\mathcal{T}^{p,r}_{k} =
(\underbrace{0,\ldots,0}_{(p+1)-\mbox{\scriptsize times}},
\underbrace{\textstyle \tau_1,\ldots ,\tau_1}_{(p-r) - \mbox{\scriptsize times}}, 
\underbrace{\textstyle \tau_2,\ldots ,\tau_2}_{(p-r) - \mbox{\scriptsize times}},\ldots, 
\underbrace{\textstyle \tau_k,\ldots ,\tau_k}_{(p-r) - \mbox{\scriptsize times}},
\underbrace{1,\ldots,1}_{(p+1)-\mbox{\scriptsize times}}),
\end{equation*}
where $k \in \N_0$, $0<\tau_i<\tau_{i+1}<1$ for all $1\leq i\leq
k-1$ and  $k$ is the number of different inner knots (in each
direction).  Thereby $r$ describes the 
resulting $C^{r}$-continuity of the space
$\mathcal{S}(\mathcal{T}^{p,r}_{k},\omega)$ at all inner knots. The
range for the regularity parameter $r$ is in general  $0
\leq r \leq p-1$. However, the focus here is on $r\geq 1$. Of course, in case of tensor-product splines, 
the knot vectors could be different in each direction. Moreover, the knot multiplicities could be different for every knot. To keep the presentation simple, we consider 
only the presented case. The spline spaces $\mathcal{S}(\mathcal{T}^{p,r}_{k},[0,1])$ 
and $\mathcal{S}(\mathcal{T}^{p,r}_{k},[0,1]^2)$ are spanned by the (tensor-product) B-splines 
$N^{p,r}_{i}$ and $N_{i,j}^{p,r}=N_{i}^{p,r} N_{j}^{p,r}$, $i,j=0, \ldots , p+k(p-r)$, respectively. Each function 
$h \in \mathcal{S}(\mathcal{T}^{p,r}_{k},[0,1])$ and $z \in \mathcal{S}(\mathcal{T}^{p,r}_{k},[0,1]^{2})$ possesses a B-spline representation
\begin{equation} \label{eq:spline_representation}
 h(t) = \sum_{i=0}^{p + k (p-r)} d_{i} N_{i}^{p,r} (t)
\end{equation}
and 
\[
 z(u,v) = \sum_{i=0}^{p + k(p-r)} \sum_{j=0}^{p + k (p-r)} d_{i,j} N_{i,j}^{p,r} (u,v)
\]
with spline control points $d_{i} \in \R$ and $d_{i,j} \in \R$, respectively.

In addition, we consider the knot vectors $\mathcal{T}^{p,r}_{k,\ell}=(t^{p,r}_{0},\ldots,t^{p,r}_{2p+2+k(p-r)})$ and 
$\mathcal{T}^{p,r}_{k,\ell,\ell'}=(t^{p,r}_{0},\ldots,t^{p,r}_{2p+3+k(p-r)})$, which are 
obtained by inserting into the knot vector $\mathcal{T}^{p,r}_{k}$ the knot $\tau_\ell$ with the index $\ell \in \{ 1,\ldots, k \}$ (in case of 
$\mathcal{T}^{p,r}_{k,\ell}$) and the knots $\tau_\ell$ and $\tau_{\ell'}$ with the indices $\ell,\ell' \in \{1, \ldots, k\} $ and $\ell \neq \ell'$ (in case 
of $\mathcal{T}^{p,r}_{k,\ell,\ell'})$. The resulting spline spaces are denoted by $\mathcal{S}(\mathcal{T}^{p,r}_{k,\ell},\omega)$ and 
$\mathcal{S}(\mathcal{T}^{p,r}_{k,\ell,\ell'},\omega)$, respectively, and are again $C^{r}$-smooth at all inner knots except at the knot $\tau_\ell$ or at 
the knots $\tau_\ell$, $\tau_{\ell'}$, respectively, where the spaces are only $C^{r-1}$-smooth. Moreover, we denote by $\mathcal{P}^{p}(\omega)$ the space of 
(tensor-product) polynomials of degree $p$ on $\omega$.

\section{$C^{1}$-smooth isogeometric spaces  and AS $G^{1}$ two-patch geometries} \label{sec:C1-smooth}

We consider a planar domain $\Omega \subset \R^{2}$ composed of two 
quadrilateral spline patches $\Omega^{(L)}$ and $\Omega^{(R)}$, i.e. 
$\Omega= \Omega^{(L)} \cup \Omega^{(R)}$, which share a whole edge as common interface $\Gamma = \Omega^{(L)} \cap \Omega^{(R)}$. We assume that each 
patch $\Omega^{(S)}$, $S \in \{L,R\}$, is the image $\ab{F}^{(S)}([0,1]^{2})$ of a regular, bijective  geometry mapping
\[
\ab{F}^{(S)}: [0,1]^2 \rightarrow \Omega^{(S)}, \quad \ab{F}^{(S)} \in \mathcal{S}(\mathcal{T}^{p,r}_{k},[0,1]^2),
\]
with the spline representations \[
\ab{F}^{(S)}(u,v)= \sum_{i=0}^{p+k(p-r)} \sum_{j=0}^{p+k(p-r)} \ab{c}^{(S)}_{i,j} N_{i,j}^{p,r} (u,v), \quad \ab{c}^{(S)}_{i,j} \in \R^{2}.
\] 
For the sake of simplicity, we assume that the two patches $\ab{F}^{(L)}$ and $\ab{F}^{(R)}$ share 
the common interface at 
\[
\ab{F}^{(L)}(0,v) = \ab{F}^{(R)}(0,v), \quad v \in [0,1].
\]
We denote the parameterization of the common curve at $\Gamma$ by $\ab{F}_{0} : [0,1] \rightarrow \R^{2}$ and assume that $\ab{F}_{0}(v)=\ab{F}^{(L)}(0,v) = 
\ab{F}^{(R)}(0,v)$. The space of isogeometric functions on $\Omega$ is given as 
\[
\mathcal{V} = \{\phi: \Omega \rightarrow \R \mbox{ such that } \phi \circ \ab{F}^{(S)} \in \mathcal{S}(\mathcal{T}^{p,r}_{k},[0,1]^2), \mbox{ }S \in \{L,R \} \}.
\]
The graph surface $\ab{\Sigma} \subset \Omega \times \R$ of an isogeometric 
function $\phi \in \mathcal{V}$ consists of the two graph surface patches 
\[
\ab{\Sigma}^{(S)} : [0,1]^{2} \rightarrow \Omega^{(S)} \times \R , \quad S \in \{L,R \},
\]
possessing the form
\[
\ab{\Sigma}^{(S)}(u,v) = (\ab{F}^{(S)}(u,v),g^{(S)}(u,v))^T,
\]
where $g^{(S)} = \phi \circ \ab{F}^{(S)} \in \mathcal{S}(\mathcal{T}^{p,r}_{k},[0,1]^2)$.
Since the geometry mappings $\ab{F}^{(S)}$, $S \in \{L,R \}$, are given, an isogeometric function $\phi \in \mathcal{V}$ is determined by the two associated spline 
functions $g^{(S)}$, $S \in \{L,R \}$, with the spline representations \[
g^{(S)}(u,v)= \sum_{i=0}^{p+k(p-r)} \sum_{j=0}^{p+k(p-r)} d^{(S)}_{i,j} N_{i,j}^{p,r} (u,v), \quad d^{(S)}_{i,j} \in \R.
\]

We are interested in $C^{1}$-smooth isogeometric functions $\phi \in
\mathcal{V}$. We assume  $C^{0}$-smoothness condition of $\phi$, that results in 
\begin{equation} \label{eq:G0condition_function}
g^{(L)}(0,v) = g^{(R)}(0,v),
\end{equation}
for $v \in [0,1]$.  We denote the function along the common
interface by  $g(v)=g^{(L)}(0,v) = g^{(R)}(0,v)$. Let us consider the space 
$\mathcal{V}^{1}$ of $C^{1}$-smooth isogeometric functions on $\Omega$, i.e. 
\[
\mathcal{V}^{1} = \mathcal{V} \cap C^{1}(\Omega),
\] 
in more detail. An isogeometric function $\phi \in \mathcal{V}$ belongs to $\mathcal{V}^{1}$ if and only if the two graph surface patches $\ab{\Sigma}^{(L)}$ and 
$\ab{\Sigma}^{(R)}$ possess a well defined tangent plane along the common interface $\ab{\Sigma}^{(L)} \cap \ab{\Sigma}^{(R)}$,  compare 
\cite{CoSaTa16, KaBuBeJu16, KaViJu15}. This is equivalent to the condition that there exist functions $\alpha^{(L)}, \alpha^{(R)}, \beta : [0,1] \rightarrow \R$ 
such that for all $v \in [0,1]$ 
\begin{equation} \label{eq:condition_alpha}
\alpha^{(L)}(v) \alpha^{(R)}(v) < 0
\end{equation}
and 
\begin{equation} \label{eq:G1condition}
\alpha^{(R)}(v) D_{u}\ab{\Sigma}^{(L)}(0,v) - \alpha^{(L)}(v)D_{u}\ab{\Sigma}^{(R)}(0,v) + \beta(v) D_{v}\ab{\Sigma}^{(R)}(0,v) = \ab{0},
\end{equation}
see \cite{Pe02}. The above described equivalent conditions are called geometric continuity of order~$1$ or $G^{1}$-smoothness (cf. \cite{HoLa93, Pe02}).

Note that the first two equations~\eqref{eq:G1condition}, i.e.
\begin{equation} \label{eq:G1condition_map}
\alpha^{(R)}(v) D_{u}\ab{F}^{(L)}(0,v) - \alpha^{(L)}(v)D_{u}\ab{F}^{(R)}(0,v) + \beta(v) D_{v}\ab{F}_{0}(v) = \ab{0},
\end{equation} 
uniquely determine the functions $\alpha^{(L)}$, $\alpha^{(R)}$ and $\beta$ up to a common function $\gamma: [0,1] \rightarrow \R$ (with $\gamma(v) \neq 0$) by 
\begin{equation} \label{eq:common_factor}
\alpha^{(L)}=\gamma(v) \bar{\alpha}^{(L)}, \alpha^{(R)}=\gamma(v) \bar{\alpha}^{(R)} \mbox{ and } \beta(v) = \gamma(v)\bar{\beta}(v),
\end{equation}
where
\begin{equation} \label{eq:alphasbar}
\bar{\alpha}^{(S)}(v) = \det ( D_{u}\ab{F}^{(S)}(0,v), D_{v}\ab{F}_{0}(v)  ), \mbox{ }S \in \{L,R \}, 
\end{equation}
and 
\begin{equation} \label{eq:betabar}
\bar{\beta}(v) = \det ( D_{u}\ab{F}^{(L)}(0,v), D_{u}\ab{F}^{(R)}(0,v) ),
\end{equation}
see e.g. \cite{CoSaTa16, Pe02}.

Therefore, an isogeometric function $\phi \in \mathcal{V}$ belongs to $\mathcal{V}^{1}$ if and only if the equation 
\begin{equation} \label{eq:G1condition_function} 
\alpha^{(R)}(v) D_{u} g^{(L)}(0,v) - \alpha^{(L)}(v)D_{u} g^{(R)}(0,v) + \beta(v) D_{v} g(v) = 0
\end{equation}
is satisfied for all $v \in [0,1]$. Equations~\eqref{eq:G0condition_function} and \eqref{eq:G1condition_function} lead to linear constraints on the spline control 
points $d_{i,j}^{(S)}$ of $g^{(S)}$, $S \in \{L, R \}$, with indices $(S,i,j)$ belonging to the index space 
\[
 I_{\Gamma} = \{ (S,i,j) \mbox{ }| \mbox{ } S \in \{L,R \}, \mbox{ } i =0,1 \mbox{ and }j = 0, \ldots, p+k(p-r) \} .
\]
Moreover, we denote by $I$ the index space formed by all spline control points $d_{i,j}^{(S)}$, $S \in \{L,R \}$, i.e.
\[
I = \{ (S,i,j) \mbox{ }| \mbox{ }S \in \{L, R \} \mbox{ and } i,j = 0, \ldots, p+k(p-r) \}.
\] 

Furthermore, for functions $\alpha^{(L)}, \alpha^{(R)}$ and $\beta$ satisfying equations~\eqref{eq:G1condition_map} there exist non-unique functions $\beta^{(L)}, 
\beta^{(R)} : [0,1] \rightarrow \R$ such that
\begin{equation} \label{eq:beta} 
\beta(v) = \alpha^{(L)}(v) \beta^{(R)}(v) - \alpha^{(R)}(v) \beta^{(L)}(v),
\end{equation}
see e.g. \cite{CoSaTa16, Pe02}.

Note that for generic patches $\ab{F}^{(L)}, \ab{F}^{(R)} \in \mathcal{S}(\mathcal{T}^{p,r}_{k},[0,1]^2)$ the functions  $\alpha^{(L)}, \alpha^{(R)}$ and $\beta$ fulfill 
$\alpha^{(L)},\alpha^{(R)} \in \mathcal{S}(\mathcal{T}^{2p-1,r-1}_{k},[0,1])$ as well as $\beta \in \mathcal{S}(\mathcal{T}^{2p,r}_{k},[0,1])$. For special 
configurations the degree may be lower and the regularity may be higher.

Motivated by \cite{CoSaTa16, Pe02}, we restrict in the following the considered geometry mappings $\ab{F}^{(L)}$ and $\ab{F}^{(R)}$ to 
geometry mappings possessing linear functions $\alpha^{(L)}$, $\alpha^{(R)}$, $\beta^{(L)}$, $\beta^{(R)}$, 
as stated in the following definition.

\begin{defn}[Analysis-suitable $G^{1}$, cf. \cite{CoSaTa16}]\label{def:asG1}
  The two-patch geometry parameterization $\ab{F}^{(L)}$ and
  $\ab{F}^{(R)}$ is  analysis-suitable $G^{1}$ (AS $G^{1}$) if there
  exist  $\alpha^{(L)}$, $\alpha^{(R)}$, $\beta^{(L)}$, $\beta^{(R)} 
\in \mathcal{P}^{1}([0,1])$, with  $\alpha^{(L)}$ and $\alpha^{(R)}$ relatively
  prime (i.e., $\deg (\gcd (\alpha^{(L)},\alpha^{(R)}) ) =0$)   satisfying equations~\eqref{eq:G1condition_map} and \eqref{eq:beta}.
\end{defn}
The condition on $\deg (\gcd (\alpha^{(L)},\alpha^{(R)}) ) =0$  was not used in \cite{CoSaTa16} but is not restrictive: if $\deg (\gcd
(\alpha^{L},\alpha^{R}) ) =1$  one can replace $\alpha^{(L)}$, $\alpha^{(R)}$ and $\beta$ by $\alpha^{(L)}/\gcd (\alpha^{(L)},\alpha^{(R)})$, $\alpha^{(R)}/\gcd 
(\alpha^{(L)},\alpha^{(R)})$ and $\beta/\gcd (\alpha^{(L)},\alpha^{(R)})$,
respectively. Obviously, the polynomial $\beta$ is divisible by $\gcd
(\alpha^{(L)},\alpha^{(R)})$, which can be seen best in
\eqref{eq:beta}. With $\deg (\gcd (\alpha^{(L)},\alpha^{(R)}) ) =0$, the 
functions $\alpha^{(L)}$, $\alpha^{(R)}$ and $\beta$ are uniquely determined up to a common 
constant.

It was shown in \cite{CoSaTa16} when the functions $\alpha^{(S)}$ and $\beta^{(S)}$, $S \in \{L,R \}$, are assumed to be of higher degree or even 
to be spline functions along the interface, then the polynomial representation along the interface is reduced to lower degrees. Hence, these spaces do not guarantee 
optimal approximation order. For this more general case the investigation of a dimension formula or of a basis construction should be possible in similar way as 
presented in the following sections, but is beyond the scope of this paper.

\section{Dimension of the space $\mathcal{V}^{1}$ for AS $G^{1}$ two-patch geometries} \label{sec:dimension}

We consider AS $G^{1}$ geometries $\ab{F}^{(L)}, \ab{F}^{(R)}
\in \mathcal{S}(\mathcal{T}^{p,r}_{k},[0,1]^2)$ with the corresponding functions  $\alpha^{(L)}$, $\alpha^{(R)} \in \mathcal{P}^{1}([0,1])$ and 
$\beta \in \mathcal{P}^{2}([0,1])$. 

Let $d_{\alpha}$ be the maximum degree of the functions $\alpha^{(S)}$, i.e. 
\begin{equation*}
d_{\alpha}=\max (\deg ( \alpha^{(L)}),\deg ( \alpha^{(R)} )).
\end{equation*}
Here we consider the actual degree of the functions. Since $\alpha^{(L)}, \alpha^{(R)} \in \mathcal{P}^{1}([0,1])$, we obtain either $d_{\alpha} = 0$ or $d_{\alpha} = 1$. 
Let $z_{\beta}$ be defined as the number of different inner knots where the function $\beta$ possesses the value zero, i.e.
\begin{equation*} 
 z_{\beta}=|\{ v_{0} \in \{\tau_1, \ldots , \tau_k \} \mbox{ } | \mbox{ } \beta(v_{0}) = 0 \}|.
\end{equation*}
Since $\beta \in \mathcal{P}^{2}([0,1])$, $ z_{\beta} \in \{ 0,1,2,k
\}$. 

Depending on $\beta$, we define a  new  knot vector
$\widetilde{\mathcal{T}^{p}_{k}}$. First, if 
${\beta}=0$ then   $\widetilde{\mathcal{T}^{p}_{k}}=
\mathcal{T}^{p,r}_{k}$.  Otherwise, assuming  ${\beta}\neq 0$, we have
three cases: if  $z_\beta =0$, we set $\widetilde{\mathcal{T}^{p}_{k}}=
\mathcal{T}^{p,r+1}_{k}$, if   $z_\beta =1$ or  $z_\beta =2$,   we set  $\widetilde{\mathcal{T}^{p}_{k}} =
\mathcal{T}^{p,r+1}_{k,\ell}$  or  $\widetilde{\mathcal{T}^{p}_{k}}= 
\mathcal{T}^{p,r+1}_{k,\ell,\ell'}$, respectively, 
where $\ell, \in \{1,\ldots,k \}$, and possibly  $ \ell' \in
\{1,\ldots,k \}$ with $\ell \neq \ell'$,  are the indices of
$\tau_{\ell}$ and $\tau_{\ell'}$, which are roots of $\beta$.

We are interested in the dimension of the isogeometric space $\mathcal{V}^{1}$. Clearly, the space $\mathcal{V}^{1}$ is the direct sum of the two subspaces 
\[
\mathcal{V}^{1}_{1}= \{\phi \in \mathcal{V}^{1} \mbox{ }| \mbox{ }d_{i,j}^{(S)}=0 \mbox{ for }S \in \{L,R \} \mbox{ and } (S,i,j) \in I_{\Gamma}  \}
\]
and 
\[
\mathcal{V}^{1}_{2}= \{\phi \in \mathcal{V}^{1} \mbox{ }| \mbox{ }d_{i,j}^{(S)}=0 \mbox{ for }S \in \{L,R \} \mbox{ and } (S,i,j) \in I \setminus I_{\Gamma}  \},
\]
which implies that 
\[
\dim \mathcal{V}^{1} = \dim \mathcal{V}^{1}_{1}  +  \dim \mathcal{V}^{1}_{2}.
\]
In \cite{KaViJu15}, the functions of $\mathcal{V}^{1}_{1}$ and $\mathcal{V}^{1}_{2}$ have been called basis functions of the first and second kind, respectively. 
We first state the dimension of $\mathcal{V}^{1}_{1}$.  
\begin{lem}
The dimension of $\mathcal{V}^{1}_{1}$ is equal to 
\[ \dim V^{1}_{1} =  2 (p+k(p-r)-1)(p+k(p-r)+1).
\]
\end{lem}
\begin{proof}
Clearly, $\dim \mathcal{V}^{1}_{1}$ is equal to the number of control points $d^{(S)}_{i,j}$, $S \in \{L,R \}$, possessing indices $(S,i,j) \in I \setminus I_{\Gamma}$, 
i.e. $\dim \mathcal{V}^{1}_1 = |I \setminus I_{\Gamma}|$.
\end{proof}

To analyze the dimension of $\mathcal{V}^{1}_{2}$, we need additional tools describing the situation at $\Gamma$. Some of these tools have been (similarly) 
introduced in \cite{CoSaTa16}. Consider the transversal 
vector $\ab{d}^{(S)} $ defined on  $\Gamma$  such that 
\[
\ab{d}^{(S)} \circ \ab{F}_{0}(v)  = (D_{u}\ab{F}^{(S)}(0,v),D_{v}\ab{F}_{0}(v)) (1,-\beta^{(S)}(v))^{T} \frac{1}{\alpha^{(S)}(v)} , \mbox{ }S \in \{ L,R \}.
\]
Observe that $\ab{d}^{(L)} = \ab{d}^{(R)} $ (which is equivalent to
\eqref{eq:G1condition_map}) and therefore we simply set 
\begin{displaymath}
  \ab{d} =  \ab{d}^{(L)} = \ab{d}^{(R)} .
\end{displaymath}

In addition, we consider the space of traces and transversal derivatives on $\Gamma$ and its pullback, which are given by
\[
\mathcal{V}^{1}_{\Gamma} = \{\Gamma \ni (x,y) \mapsto (\phi(x,y) , \nabla \phi(x,y) \cdot \ab{d}(x,y)) , \mbox{ such that }\phi \in \mathcal{V}^{1}_{2} \}
\]
and
\[
\widehat{\mathcal{V}^{1}_{\Gamma}} = \{(\phi,\nabla \phi \cdot \ab{d}) \circ \ab{F}_{0}, \mbox{ such that } \phi \in \mathcal{V}^{1}_{2}\},
\]
respectively. The transversal vector $\ab{d}$, $\mathcal{V}^{1}_{\Gamma}$ and $\widehat{\mathcal{V}^{1}_{\Gamma}}$ depend only on the choice of $\beta^{(L)}$ and 
$\beta^{(R)}$, since $\alpha^{(L)}$ and $\alpha^{(R)}$ are now
uniquely determined by the geometry mappings $\ab{F}^{(L)}$ and
$\ab{F}^{(R)}$. Associated to $\ab{d}$, we consider the transversal derivative of $\phi$ with respect to $\ab{d}$ on $\Gamma$, i.e. $(\nabla \phi \cdot \ab{d}) \circ 
\ab{F}_{0} $.

Clearly, for $\phi \in \mathcal{V}_{2}^{1}$ the function $\phi \circ \ab{F}_{0}$ is a spline function. More precisely, we have:
\begin{lem} \label{lem:space_g0}
If $\phi \in \mathcal{V}^{1}_{2}$, then $\phi \circ \ab{F}_{0} \in \mathcal{S}(\widetilde{\mathcal{T}^{p}_{k}},[0,1])$. 
\end{lem}
\begin{proof}
Analyzing equation~\eqref{eq:G1condition_function}, we observe that $\phi \circ \ab{F}_{0} \in \mathcal{S}(\widetilde{\mathcal{T}^{p}_{k}},[0,1])$, since 
$D_{u} g^{(S)}(0,v) \in \mathcal{S}(\mathcal{T}^{p,r}_{k},[0,1])$.
\end{proof}

The following lemma ensures that for $\phi \in 
\mathcal{V}_{2}^{1}$ the function $(\nabla \phi \cdot \ab{d}) \circ \ab{F}_{0}$ is a spline function, too.  

\begin{lem}\label{lem:space_g1}
 If $\phi \in \mathcal{V}^{1}_{2}$, then $(\nabla \phi \cdot \ab{d}) \circ \ab{F}_{0} \in \mathcal{S}(\mathcal{T}^{p-d_{\alpha},r-1}_{k},[0,1])$.
\end{lem}
\begin{proof}
Recall that $\ab{d} \circ \ab{F}_{0}(v) = \ab{d}^{(L)} \circ \ab{F}_{0}(v) =  \ab{d}^{(R)} \circ \ab{F}_{0}(v) $. An isogeometric function $\phi$ belong to the space 
$\mathcal{V}^{1}_{2}$ if and only if
\[
 (\nabla^{(L)} \phi \cdot \ab{d}) \circ \ab{F}_{0}(v) = (\nabla^{(R)} \phi \cdot \ab{d}) \circ \ab{F}_{0}(v) = (\nabla \phi \cdot \ab{d}) \circ \ab{F}_{0}(v) 
\]
for all $v \in [0,1]$. Since 
\begin{equation} \label{eq:rel_g1}
 (\nabla^{(S)} \phi \cdot \ab{d} ) \circ \ab{F}_{0}(v) = \frac{D_{u}g^{(S)}(0,v)-\beta^{(S)}(v)D_{v}g(v)}{\alpha^{(S)}(v)}, \mbox{ }S \in \{L,R \},
\end{equation}
we obtain that $\phi \in \mathcal{V}_{2}^{1}$ if and only if
\begin{equation} \label{eq:G1condition_function_version2}
 (D_{u}g^{(L)}(0,v)-\beta^{(L)}(v)D_{v}g(v))\alpha^{(R)}(v) =  (D_{u}g^{(R)}(0,v)-\beta^{(R)}(v)D_{v}g(v))\alpha^{(L)}(v) 
\end{equation}
for all $v \in [0,1]$. (Condition~\eqref{eq:G1condition_function_version2} is exactly the same as condition~\eqref{eq:G1condition_function} by substituting $\beta$ 
via~\eqref{eq:beta}.) Recall that $\deg (\gcd (\alpha^{L},\alpha^{R}) ) =0$. Therefore, by dividing Equation~\eqref{eq:G1condition_function_version2} by $\alpha^{(L)}$, 
we see that $(\nabla^{(L)} \phi \cdot \ab{d} ) \circ \ab{F}_{0} \in \mathcal{S}(\mathcal{T}^{p-\deg(\alpha^{(L)}),r-1}_{k},[0,1])$. Analogously we can show that 
$(\nabla^{(R)} \phi \cdot \ab{d} ) \circ \ab{F}_{0} \in \mathcal{S}(\mathcal{T}^{p-\deg (\alpha^{(R)}),r-1}_{k},[0,1])$ and obtain that 
$(\nabla \phi \cdot \ab{d} ) \circ \ab{F}_{0} \in \mathcal{S}(\mathcal{T}^{p-d_{\alpha},r-1}_{k},[0,1])$.
\end{proof}

There is a one-to-one correspondence between
 trace and transversal derivative at $\Gamma$, and $\phi \in
\mathcal{V}^{1}_{2}$. 
\begin{prop} \label{prop:representation} 
For any $(g_{0},g_{1}) \in \widehat{\mathcal{V}^{1}_{\Gamma}}$ there
exists a unique $\phi \in \mathcal{V}^{1}_{2}$ such that
$(g_{0},g_{1})  = (\phi,\nabla \phi \cdot \ab{d}) \circ \ab{F}_{0}$,
given, for $S \in \{L,R \}$, by
\begin{equation}
  \label{eq:phi-from-g0-and-g1}
  \phi \circ \ab{F}^{(S)} = g_{0}(v)(N^{p,r}_{0}(u) + N^{p,r}_{1}(u)) + \left( \alpha^{(S)}(v) g_{1}(v) + \beta^{(S)}(v) 
  g_{0}^\prime(v) \right)\frac{1}{(k+1)p}N^{p,r}_{1}(u) .
\end{equation}
\end{prop}
\begin{proof}
Recall the notation   $\phi \circ \ab{F}^{(S)} = g^{(S)}(u,v) $. Equation~\eqref{eq:rel_g1} is equivalent to 
\begin{equation} \label{eq:rel_gu}
D_{u} g^{(S)}(0,v) =\alpha^{(S)} (v) (\nabla^{(S)} \phi \cdot \ab{d} ) \circ \ab{F}_{0}(v) + \beta^{(S)}(v) g_{0}^\prime(v)=\alpha^{(S)}(v) g_{1}(v) + 
\beta^{(S)}(v) g_{0}^\prime(v).
\end{equation}
By Taylor expansion of $g^{(S)}(u,v)$ and using equation~\eqref{eq:rel_gu}, we obtain
\begin{eqnarray*}
g^{(S)}(u,v) & = & g^{(S)}(0,v) + (D_{u} g^{(S)})(0,v) u + \mathcal{O}(u^{2}) \\
 & = & g_{0}(v) + (\alpha^{(S)}(v) g_{1}(v) + \beta^{(S)}(v)g_{0}^\prime(v) )u + \mathcal{O}(u^{2}) \\
 & = &  g_{0}(v)(N^{p,r}_{0}(u) + N^{p,r}_{1}(u)) + (\alpha^{(S)}(v) g_{1}(v) + \beta^{(S)}(v) g_{0}^\prime(v) )\frac{1}{(k+1)p}N^{p,r}_{1}(u),
\end{eqnarray*}
for $S \in \{L, R\}$. 
\end{proof}
Two useful corollaries of Proposition \ref{prop:representation}
follow below. Both hold for any possible choice of $\beta^{(L)}$ and $\beta^{(R)}$.
\begin{cor}\label{cor:space-dimensions-at-Gamma}
The dimension of $\mathcal{V}^{1}_{2}$ is  equal to the dimension of $\widehat{\mathcal{V}^{1}_{\Gamma}}$.
\end{cor}
\begin{cor}\label{cor:characterization-of-C1-functions}
 It holds that $(g_{0},g_{1}) \in \widehat{\mathcal{V}^{1}_{\Gamma}} $ if
 and only if:
 \begin{eqnarray}
   g_{0}&\in&
              \mathcal{S}(\widetilde{\mathcal{T}^{p}_{k}},[0,1])\label{eq:C1-char-g0}\\
 g_{1}&\in& \mathcal{S}({\mathcal{T}^{p-d_{\alpha},r-1}_{k}},[0,1]) \label{eq:C1-char-g1}\\
\alpha^{(S)} g_{1} + \beta^{(S)}g_{0}^\prime &\in&    \mathcal{S}({\mathcal{T}^{p,r}_{k}},[0,1]),
                                                \text{ for $S \in \{L,R \}$.}\label{eq:C1-char-mixed}
 \end{eqnarray}
\end{cor}
\begin{proof}
 The statement follows from Lemma \ref{lem:space_g0}--\ref{lem:space_g1} and
  \eqref{eq:phi-from-g0-and-g1} in Proposition \ref{prop:representation}.
\end{proof}
To investigate the dimension of $\widehat{\mathcal{V}^{1}_{\Gamma}}$, we consider the spaces
\[
\Gamma_{0} = \{ g_{0} \mbox{ such that }  (g_{0}, g_{1}) \in
\widehat{\mathcal{V}^{1}_{\Gamma}}  \text{ for some }  g_{1}  \}
\equiv \{ \phi \circ \ab{F}_{0}, \mbox{ such that } \phi \in \mathcal{V}^{1}_{2} \}
\]
and 
\[
\Gamma_{1} = \{ (0, g_{1}) \in \widehat{\mathcal{V}^{1}_{\Gamma}} \}.
\]
Clearly, $\dim \widehat{\mathcal{V}^{1}_{\Gamma}} = \dim \Gamma_0 + \dim \Gamma_{1}$. The following two lemmas state the dimension of $\Gamma_{0}$ and $\Gamma_{1}$.

\begin{lem}
It holds that 
 \begin{equation}
   \label{eq:Gamma0-characterization}
\Gamma_{0} =   \mathcal{S}(\widetilde{\mathcal{T}^{p}_{k}},[0,1])
 \end{equation} 
and consequently 
\[
\dim \Gamma_{0} = p + k (p-r-1) + 1 +  z_{\beta}.
\]
\end{lem}
\begin{proof}
Thanks to Corollary \ref{cor:characterization-of-C1-functions}, we
need to  show that for all 
$g_{0} \in \mathcal{S}(\widetilde{\mathcal{T}^{p}_{k}},[0,1])$ we can
construct $ g_{1} $ such that \eqref{eq:C1-char-g1} and
\eqref{eq:C1-char-mixed} holds.

 In case of $\beta \equiv 0$, and also when $\beta \neq 0$ with $z_\beta =0$, we can simply set  $g_{1}(v)=0$. 
In the rest of the proof, we always assume  $\beta \neq 0$. 

In case of $z_\beta =1$, we choose
$g_{1}(v)=-\frac{\beta^{(L)}(\tau_{\ell})}{\alpha^{(L)}(\tau_{\ell})}
g'_{0}(v) \in  \mathcal{S}({\mathcal{T}^{p-1,r-1}_{k}},[0,1])$, where $\tau_{\ell}$ is a root of $\beta$.  We need to prove
\eqref{eq:C1-char-mixed}, that is, to show that 
\begin{equation}
  \label{eq:temp-1}
  \left(\beta^{(S)}-\alpha^{(S)} \frac{\beta^{(L)}(\tau_{\ell})}{\alpha^{(L)}(\tau_{\ell})} \right)g'_{0},
\end{equation}
is $C^{r}$-smooth, for $S \in \{L,R \}$. We only need to check the
$r$-th derivative of \eqref{eq:temp-1}, for $S \in \{L,R \}$, which
is
\[
r \left(\beta^{(S)}\mbox{}'(v) -\alpha^{(S)}\mbox{}'(v)
  \frac{\beta^{(L)}(\tau_{\ell})}{\alpha^{(L)}(\tau_{\ell})}
\right)g^{(r)}_{0}(v ) +  \left(\beta^{(S)}\mbox{}(v) -\alpha^{(S)}\mbox{}(v)
  \frac{\beta^{(L)}(\tau_{\ell})}{\alpha^{(L)}(\tau_{\ell})}
\right)g^{(r+1)}_{0}(v ), 
\]
that is continuous in $[0,1] \setminus \tau_{\ell} $ for  the regularity of $g_{0}
$ and continuous when $v \rightarrow \tau_{\ell}$ since the second
addendum vanishes in the limit. 

In case of $z_\beta =2$, we choose $g_{1}(v)=-(\frac{\beta^{(L)}(\tau_{\ell})}{\alpha^{(L)}(\tau_{\ell})}
\hat{g}'_{0}(v) + \frac{\beta^{(L)}(\tau_{\ell'})}{\alpha^{(L)}(\tau_{\ell'})}
\tilde{g}'_{0}(v)) \in  \mathcal{S}({\mathcal{T}^{p-1,r-1}_{k}},[0,1])$, where 
$\hat{g}_{0} \in \mathcal{S}({\mathcal{T}^{p,r+1}_{k,\ell}},[0,1])$ and $\tilde{g}_{0} \in \mathcal{S}({\mathcal{T}^{p,r+1}_{k,\ell'}},[0,1])$ are non-unique functions 
such that $g_{0} = \hat{g}_{0} + \tilde{g}_{0}$, and $\tau_{\ell}, \tau_{\ell'}$ with the indices $\ell,\ell' \in \{1, \ldots , k \}$ and $\ell \neq \ell'$ are the two 
roots of $\beta$. As before, we need to prove \eqref{eq:C1-char-mixed}, that is, to show that 
\begin{equation}
  \label{eq:temp-3}
  \left(\beta^{(S)}-\alpha^{(S)} \frac{\beta^{(L)}(\tau_{\ell})}{\alpha^{(L)}(\tau_{\ell})} \right)\hat{g}'_{0} + \left(\beta^{(S)}-\alpha^{(S)} 
  \frac{\beta^{(L)}(\tau_{\ell'})}{\alpha^{(L)}(\tau_{\ell'})} \right)\tilde{g}'_{0},
\end{equation}
is $C^{r}$-smooth, for $S \in \{L,R \}$. We only need to check again that the $r$-th derivative of \eqref{eq:temp-3}, for $S \in \{L,R \}$, is continuous 
in $[0,1]$, which can be done analogous to the case $z_{\beta}=1$. The dimension of $\Gamma_0$ follows directly from the definition of the spline space. This concludes 
the proof.
\end{proof}

\begin{lem}
It holds that
 \begin{equation}
   \label{eq:Gamma1-characterization}
\Gamma_{1} =  \{ 0 \} \times  \mathcal{S}(\mathcal{T}^{p-d_{\alpha},r}_{k})
 \end{equation} 
and consequently
\[
\dim \Gamma_{1} = p + k (p-r-1) + (1 - d_{\alpha})(k+1).
\]
\end{lem}
\begin{proof}
Thanks to Corollary \ref{cor:characterization-of-C1-functions},  $(0,g_1) \in \Gamma_{1}$ if and only if $
g_{1} \in   \mathcal{S}(\mathcal{T}^{p-d_{\alpha},r}_{k} )$. 
The dimension of $\Gamma_1$ follows directly from the definition of the spline space. 
\end{proof}

This leads to the following result.
\begin{lem}
The dimension of $\mathcal{V}^{1}_{2}$ is equal to
\[
\dim \mathcal{V}^{1}_{2}= 2(p+k(p-r-1))+1 + (1-d_{\alpha})(k+1) +  z_{\beta}.
\]
\end{lem}
Finally, we obtain:
\begin{thm}
The dimension of $\mathcal{V}^{1}$ is equal to
\[ 
\dim \mathcal{V}^{1} = \underbrace{2 (p+k(p-r)-1)(p+k(p-r)+1)}_{\dim \mathcal{V}^{1}_{1}} + \underbrace{2(p+k(p-r-1))+1 + (1-d_{\alpha})(k+1) +  
z_{\beta}}_{\dim \mathcal{V}^{1}_{2}}.
\]
\end{thm}

\begin{rem}
Our results are in agreement with those in \cite{KaViJu15}, where the special case of bilinear geometry mappings $\ab{F}^{(L)}$ and $\ab{F}^{(R)}$ with $d_{\alpha}=1$, 
$z_{\beta}=0$, $p=3,4$ and $r=1$ was considered.
\end{rem}

\section{Basis of the space $\mathcal{V}^{1}$ for AS $G^{1}$ two-patch geometries} \label{sec:basis}

We present an explicit basis construction for the space $\mathcal{V}^{1}$ for AS $G^{1}$ two-patch geometries. Our basis will consist of a basis for the 
space $\mathcal{V}^{1}_{1}$ and of a basis for the space $\mathcal{V}^{1}_{2}$.  

\subsection{Basis of $\mathcal{V}^{1}_{1}$}

The functions in $\mathcal{V}^{1}_{1}$ are not influenced by the interface $\Gamma$. 
Hence, a basis of $\mathcal{V}^{1}_{1}$ can be constructed in a straightforward way from the standard basis on a single patch. Consider for 
$i=2, \ldots, p+k(p-r)$, $j=0,\ldots, p+k(p-r)$, the isogeometric functions 
$\phi^{(L)}_{i,j}$ and $\phi^{(R)}_{i,j}$ determined by
\[
\phi^{(L)}_{i,j}: \mbox{ } g^{(L)}(u,v)=N_{i,j}^{p,r}(u,v), \mbox{ }g^{(R)}(u,v)=0 \mbox{ and }
\phi^{(R)}_{i,j}: \mbox{ }g^{(L)}(u,v)=0, \mbox{ } g^{(R)}(u,v)=N_{i,j}^{p,r}(u,v).
\]
Then the collection of isogeometric functions 
$\{ \phi^{(S)}_{i,j} \}_{(S,i,j) \in I \setminus I_{\Gamma}}$ forms a basis of the space $\mathcal{V}^{1}_{1}$. Note that these basis functions do not depend 
on the geometry mappings $\ab{F}^{(L)}$ and $\ab{F}^{(R)}$ (and therefore do not depend on $\alpha^{(L)}$, $\alpha^{(R)}$ and $\beta$, too). This is in contrast to 
the basis functions of $\mathcal{V}^{1}_{2}$, see Section~\ref{subsec:basis_second}.  

\subsection{Basis of $\mathcal{V}^{1}_{2}$} \label{subsec:basis_second}

We present a construction of a basis of the space $\mathcal{V}^{1}_{2}$. Thereby, the resulting basis functions depend on $\alpha^{(L)}$, $\alpha^{(R)}$, $\beta^{(L)}$ 
and $\beta^{(R)}$, and hence on the geometry mappings $\ab{F}^{(L)}$ and $\ab{F}^{(R)}$. The idea is to generate a basis of $\mathcal{V}^{1}_{2}$ by means of a basis of 
$\Gamma_{0}$ and of a basis of $\Gamma_{1}$ (technically by means of a basis of $\widehat{\mathcal{V}^{1}_{\Gamma}}$), since Proposition~\ref{prop:representation} 
provides an explicit representation for the desired basis functions of $\mathcal{V}^{1}_{2}$, given by
\begin{equation} \label{eq:representation}
 g^{(S)}(u,v)= g_{0}(v)(N^{p,r}_{0}(u) + N^{p,r}_{1}(u)) + \left( \alpha^{(S)}(v) g_{1}(v) + \beta^{(S)}(v) g_{0}^\prime(v) \right)\frac{1}{(k+1)p}N^{p,r}_{1}(u), 
 \end{equation}  
for $S \in \{L,R \}$. More precisely, the construction works as follows:

\begin{enumerate}
 \item We select that pair of functions $\beta^{(L)}$ and $\beta^{(R)}$, such that \eqref{eq:beta} holds and which minimizes the term
 \begin{equation} \label{eq:min_beta}
  ||\beta^{(L)} ||^{2}_{L_{2}([0,1])} +  ||\beta^{(R)} ||^{2}_{L_{2}([0,1])}.
 \end{equation}
In case of ${\beta} = 0$, we have $\beta^{(L)}(v) = \beta^{(R)}(v)=0$.

\item Let $\tilde{n} = \dim \mathcal{S}(\widetilde{\mathcal{T}^{p}_{k}},[0,1])$.
Depending on $\beta$, we first choose for the space $\mathcal{S}(\widetilde{\mathcal{T}^{p}_{k}},[0,1])$ a basis $\tilde{N}_{i}$, $i=1, \ldots, \tilde{n}-1$, and then 
for each pair $(g_{0},g_{1})= (\tilde{N}_{i},\tilde{g}_{1,i})$, $i =0,\ldots, \tilde{n}-1$,  the function $\tilde{g}_{1,i}$ as follows: 
\begin{itemize}
\item \emph{Case ${\beta} =0 \text{ or } z_{\beta} = 0$}: The functions $\tilde{N}_{i}$, $i=1, \ldots, \tilde{n}-1$, are the B-splines of 
$\mathcal{S}(\widetilde{\mathcal{T}^{p}_{k}},[0,1])$, and $\tilde{g}_{1,i}(v)=0$.
\item \emph{Case ${\beta} \neq 0$ \text{ and } $z_{\beta}=1$:} Let $\tau_{\ell}$ with the index $\ell \in \{1, \ldots, k \}$ be the root of $\beta$. The functions 
$\tilde{N}_{i}$, $i=1, \ldots, \tilde{n}-2$, are the B-splines of $\mathcal{S}(\mathcal{T}_{k}^{p,r+1},[0,1])$ and the function $\tilde{N}_{\tilde{n}-1}$ is one of the 
B-splines of $\mathcal{S}(\mathcal{T}_{k,\ell}^{p,r+1},[0,1])$ with the property $\tilde{N}_{\tilde{n}-1}(\tau_{\ell}) \neq 0$. 
The function $\tilde{g}_{1,i}$ is given by
 \[
  \tilde{g}_{1,i}(v) = \left\{ \begin{array}{cl}
                     -\frac{\beta^{(L)}(\tau_{\ell})}{\alpha^{(L)}(\tau_{\ell})} \tilde{N}'_{i}(v) & \mbox{if } i=\tilde{n}-1, \\
                     0 & \mbox{otherwise}.
                    \end{array}
                    \right.
 \]
\item \emph{Case ${\beta} \neq 0$ \text{ and }$z_{\beta}=2$:} Let $\tau_{\ell},\tau_{\ell'}$ with the indices $\ell,\ell' \in \{1, \ldots, k \}$ and $\ell \neq \ell'$ 
be the two roots of $\beta$. The functions $\tilde{N}_{i}$, $i=1, \ldots, \tilde{n}-3$, are the B-splines of $\mathcal{S}(\mathcal{T}_{k}^{p,r+1},[0,1])$, the 
function $\tilde{N}_{\tilde{n}-2}$ is one of the B-splines of $\mathcal{S}(\mathcal{T}_{k,\ell}^{p,r+1},[0,1])$ with the property 
$\tilde{N}_{\tilde{n}-2}(\tau_{\ell}) \neq 0$, and the function $\tilde{N}_{\tilde{n}-1}$ is one of the B-splines of $\mathcal{S}(\mathcal{T}_{k,\ell'}^{p,r+1},[0,1])$ 
with the property $\tilde{N}_{\tilde{n}-1}(\tau_{\ell'}) \neq 0$. The function $\tilde{g}_{1,i}$ is given by
 \[
  \tilde{g}_{1,i}(v) = \left\{ \begin{array}{cl}
                     -\frac{\beta^{(L)}(\tau_{\ell})}{\alpha^{(L)}(\tau_{\ell})} \tilde{N}'_{i}(v) & \mbox{if } i=\tilde{n}-2, \\
                      -\frac{\beta^{(L)}(\tau_{\ell'})}{\alpha^{(L)}(\tau_{\ell'})} \tilde{N}'_{i}(v) & \mbox{if } i=\tilde{n}-1, \\
                     0 & \mbox{otherwise}.
                    \end{array}
                    \right.
 \]
\end{itemize}
Then each pair $(g_{0},g_{1})= (\tilde{N}_{i},\tilde{g}_{1,i})$, $i=0,\ldots,\tilde{n}-1$, determines a basis function $\phi_{0,i}$ 
 via representation~\eqref{eq:representation}, i.e. 
 \begin{equation*} 
  \phi_{0,i}: \mbox{ }g^{(S)}(u,v) = \tilde{N}_{i}(v)(N^{p,r}_{0}(u) + N^{p,r}_{1}(u)) + ( \alpha^{(S)}(v) \tilde{g}_{1,i}(v) + 
  \beta^{(S)}(v) \tilde{N}'_{i}(v) )\frac{1}{(k+1)p}N^{p,r}_{1}(u),
 \end{equation*}
for $S \in \{L,R \}$. In case of ${\beta}\neq 0 $, the representation of $\phi_{0,i}$ simplifies to 
 \begin{equation} \label{eq:simplified_representation}
  \phi_{0,i}: \mbox{ }g^{(S)}(u,v) = \tilde{N}_{i}(v)(N^{p,r}_{0}(u) + N^{p,r}_{1}(u)) +  
  \beta^{(S)}(v) \tilde{N}'_{i}(v) \frac{1}{(k+1)p}N^{p,r}_{1}(u), \mbox{ } S \in \{L,R \},
 \end{equation}
 except for the index $i=\tilde{n}-1$ if $z_{\beta}=1$ and for the index $i=\tilde{n}-2$ or $i=\tilde{n}-1$ if $z_{\beta}=2$. In case of $\beta= 0$, the representation 
 of $\phi_{0,i}$ even simplifies to
 \[
  \phi_{0,i}: \mbox{ }g^{(S)}(u,v) = N_{i}^{p,r}(v)(N^{p,r}_{0}(u) + N^{p,r}_{1}(u)), \mbox{ } S \in \{L,R \}.
 \]
\item Let $\bar{n} = \dim \mathcal{S}(\mathcal{T}^{p-d_{\alpha},r}_{k},[0,1])$, and let $\bar{N}_{j}$, $j =0,\ldots, \bar{n}-1$, be the B-spline basis functions of the 
space $\mathcal{S}(\mathcal{T}^{p-d_{\alpha},r}_{k},[0,1])$. For each pair $(g_{0},g_{1})= (0,\bar{N}_{j})$,  $j =0,\ldots, \bar{n}-1$, a basis function $\phi_{1,j}$ is 
defined via representation~\eqref{eq:representation}, i.e.
\[
 \phi_{1,j}: \mbox{ }g^{(S)}(u,v) = \frac{1}{(k+1)p} \alpha^{(S)}(v) \bar{N}_{j}(v) N^{p,r}_{1}(u), \mbox{ }S \in \{ L, R \}.
\]
\end{enumerate}
Then the collection of isogeometric functions $\{\phi_{0,i}, \phi_{1,j}  \}_{i=0,\ldots,\tilde{n}-1;j=0,\ldots,\bar{n}-1}$ forms a basis of the space 
$\mathcal{V}^{1}_{2}$. 

\begin{rem}
The basis functions $\phi_{0,i}$ and $\phi_{1,j}$ have small local supports which are comparable with the supports of the basis functions considered 
in \cite{KaBuBeJu16,KaViJu15}, which were constructed for the special case of bilinear geometry mappings $\ab{F}^{(L)}$ and $\ab{F}^{(R)}$ with 
$d_{\alpha}=1$, $z_{\beta}=0$, $p=3,4$, $r=1$ and $\tau_{i}=\frac{i}{k+1}$ for $i=1,\ldots,k$. 
\end{rem}

\begin{rem}
Our selection of the functions $\beta^{(L)}$ and $\beta^{(R)}$, as described above, is of course only one possibility. It would be even possible to choose for each 
function $\phi_{0,i}$ a different pair of functions $\beta^{(L)}$ and $\beta^{(R)}$, if desired. In addition, in case of $\beta \neq 0$ and $z_{\beta}=1$, the choice of 
the functions $\beta^{(L)}$ and $\beta^{(R)}$ satisfying $\beta^{(L)}(\tau_{\ell})=\beta^{(R)}(\tau_{\ell})=0$, where $\tau_{\ell}$ with the index 
$\ell \in \{1,\ldots , k \}$ is a root of $\beta$, would also lead for this case to the simplified representation \eqref{eq:simplified_representation} for all 
functions $\phi_{0,i}$.
\end{rem}

\begin{ex} \label{ex:functions} 
We consider the AS $G^{1}$ two-patch geometry~$\ab{F}^{(L)}$, $\ab{F}^{(R)} \in \mathcal{P}^{3}([0,1]^{2})$, which is shown in Fig.~\ref{fig:ex_functions} and is 
given by the control points
\[ \small \ab{c}_{0,0}^{(L)} = (\frac{3}{50}, -\frac{1}{20}), \quad
\ab{c}_{1,0}^{(L)}= -(\frac{774887}{668100}, \frac{19799}{267240}), \quad 
\ab{c}_{2,0}^{(L)} = -(\frac{189}{100}, \frac{107}{100}), \quad
\ab{c}_{3,0}^{(L)} =-(\frac{151}{50}, \frac{49}{100}), 
\]
\[ \small \ab{c}_{0,1}^{(L)} = (\frac{7}{20}, \frac{24}{25}), \quad
\ab{c}_{1,1}^{(L)}= (-\frac{294947}{334050}, \frac{819}{850}), \quad 
\ab{c}_{2,1}^{(L)} = (-\frac{101}{50}, \frac{57}{100}), \quad
\ab{c}_{3,1}^{(L)} =(-\frac{67}{25}, \frac{4}{5}), 
\]
\[ \small \ab{c}_{0,2}^{(L)} = (\frac{47}{100}, 2), \quad
\ab{c}_{1,2}^{(L)}= (-\frac{623281}{801720}, \frac{233057}{117900}), \quad 
\ab{c}_{2,2}^{(L)} = (-\frac{213}{100}, \frac{109}{50}), \quad
\ab{c}_{3,2}^{(L)} =(-\frac{189}{50}, \frac{203}{100}), 
\]
\[ \small \ab{c}_{0,3}^{(L)} = (\frac{1}{20}, \frac{307}{100}), \quad
\ab{c}_{1,3}^{(L)}= (-\frac{422117}{334050}, \frac{1969021}{668100}), \quad 
\ab{c}_{2,3}^{(L)} = (-\frac{201}{100}, \frac{189}{50}), \quad
\ab{c}_{3,3}^{(L)} =(-\frac{84}{25}, \frac{331}{100}), 
\]
and 
\[ \small \ab{c}_{0,0}^{(R)} = \ab{c}_{0,0}^{(L)}, \quad
\ab{c}_{1,0}^{(R)}= (\frac{787217}{801720}, -\frac{50021}{400860}), \quad 
\ab{c}_{2,0}^{(R)}= (\frac{123}{50}, -\frac{61}{100}), \quad
\ab{c}_{3,0}^{(R)}= (\frac{347}{100}, -\frac{6}{25}), 
\]
\[ \small \ab{c}_{0,1}^{(R)} = \ab{c}_{0,1}^{(L)}, \quad
\ab{c}_{1,1}^{(R)}= (\frac{3705053}{3006450}, \frac{2796793}{3006450}), \quad 
\ab{c}_{2,1}^{(R)}= (\frac{113}{50}, \frac{17}{20}), \quad
\ab{c}_{3,1}^{(R)}= (\frac{53}{20}, \frac{113}{100}), 
\]
\[ \small \ab{c}_{0,2}^{(R)} = \ab{c}_{0,2}^{(L)} , \quad
\ab{c}_{1,2}^{(R)}= (\frac{581369}{445400}, \frac{24743903}{12025800}), \quad 
\ab{c}_{2,2}^{(R)}= (2, \frac{107}{50}), \quad
\ab{c}_{3,2}^{(R)}= (\frac{351}{100}, \frac{9}{4}), 
\]
\[ \small \ab{c}_{0,3}^{(R)} = \ab{c}_{0,3}^{(L)}, \quad
\ab{c}_{1,3}^{(R)}= (\frac{267523}{334050}, \frac{1298303}{400860}, \quad 
\ab{c}_{2,3}^{(R)}= (\frac{107}{50}, \frac{319}{100}), \quad
\ab{c}_{3,3}^{(R)}= (\frac{297}{100}, \frac{173}{50}).
\]
The corresponding functions $\alpha^{(L)}$, $\alpha^{(R)}$ and $\beta$ are given by means of Equations~\eqref{eq:common_factor}-\eqref{eq:betabar} and selecting the 
function $\gamma$ as
\[
\gamma(v)= \frac{10000}{8167 + 60 v - 407 v^2 + 516 v^3 + 407 v^4},
\] 
which leads to 
\[
 \alpha^{(L)}(v)= -\frac{3}{2} (9 +v), \mbox{ }  \alpha^{(R)}(v)= -\frac{3}{2} (-7 +v) \mbox{ and } \beta(v)= \frac{1}{12} (15 - 32v + v^{2}),
\]
respectively. The minimization of~\eqref{eq:min_beta} leads to
\[
\beta^{(L)}(v) = -\frac{83}{1194} + \frac{503 v}{3528} \mbox{ and } \beta^{(R)}(v) = -\frac{23}{597} + \frac{152 v}{1791}.  
\]
Clearly, we have  $\ab{F}^{(L)},\ab{F}^{(R)} \in \mathcal{S}(\mathcal{T}^{3,1}_{k},[0,1]^{2})$ with $k \geq 0$, and we obtain for these selections of the geometry 
mappings $\ab{F}^{(L)}$ and $\ab{F}^{(R)}$ that
\[
\dim \mathcal{V}^{1} = 2(2+2k)(4+2k)+7+2k= 23 + 26 k + 8 k^2.
\]
Below, let $\mathcal{T}^{3,1}_{k}$ be the uniform knot vector 
\[
\mathcal{T}^{3,1}_{k}=(0,0,0,0,\frac{1}{k+1},\frac{1}{k+1},\frac{2}{k+1},\frac{2}{k+1},\ldots,\frac{k}{k+1},\frac{k}{k+1},1,1,1,1). \]
Fig.~\ref{fig:ex_functions} shows the graphs of the resulting basis functions $\phi_{0,i}$ and $\phi_{1,j}$ of the space $\mathcal{V}^{1}_{2}$ for $k=2$. 

\begin{figure}[htp]
\centering\footnotesize
\begin{tabular}{ccc}
\includegraphics[width=4.1cm,clip]{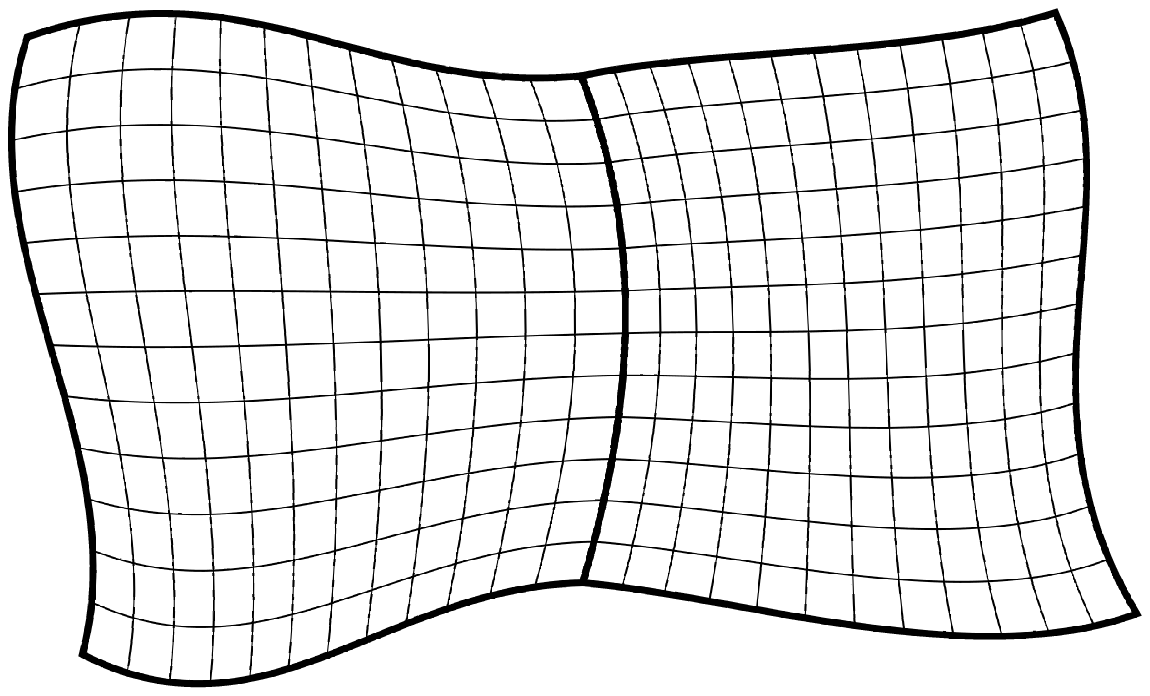}  &
\includegraphics[width=4.1cm,clip]{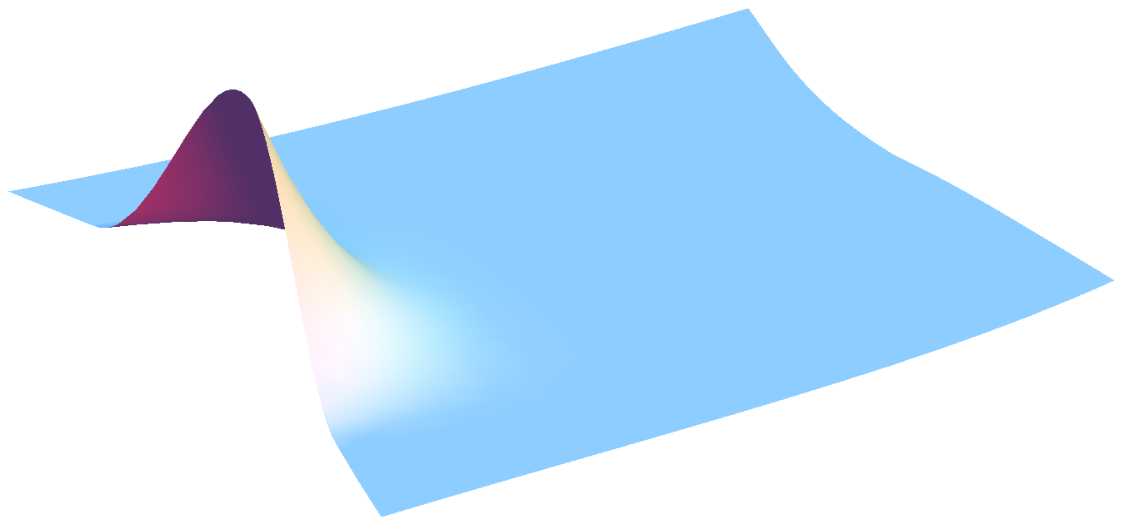} &
\includegraphics[width=4.1cm,clip]{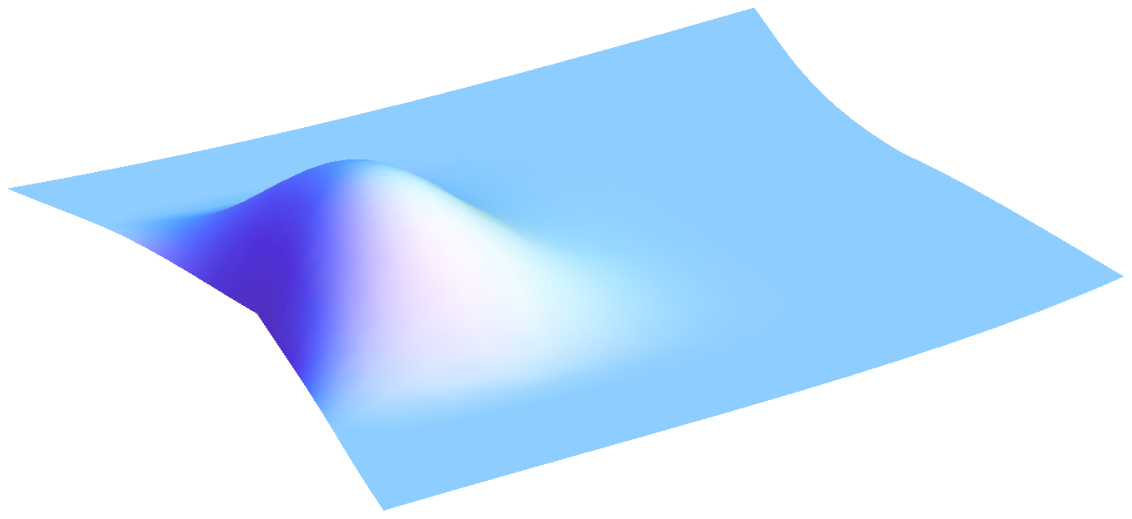} \\
AS $G^{1}$ two-patch geometry & $\phi_{0,0}$ & $\phi_{0,1}$ \\
\includegraphics[width=4.1cm,clip]{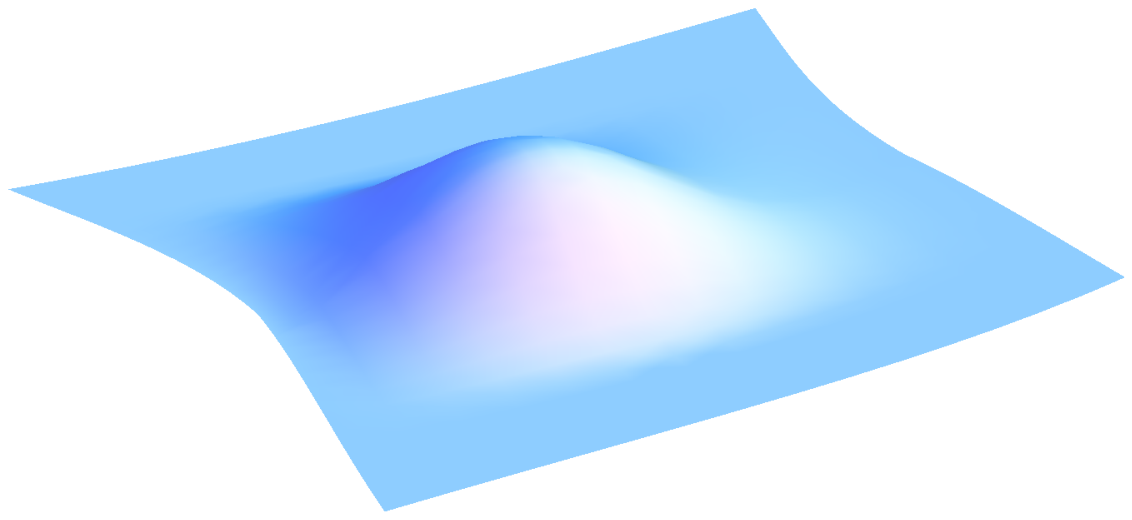}  &
\includegraphics[width=4.1cm,clip]{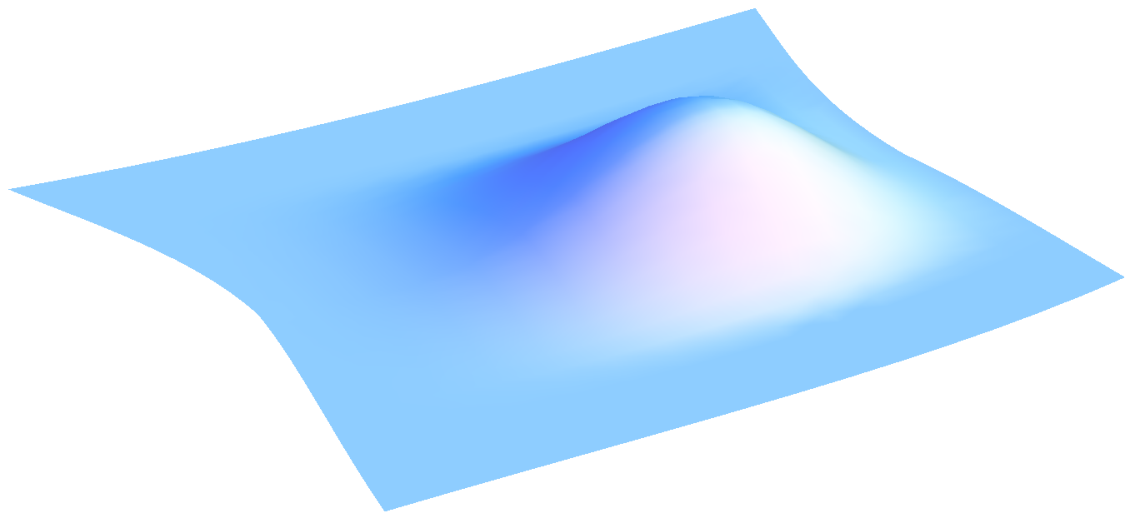} &
\includegraphics[width=4.1cm,clip]{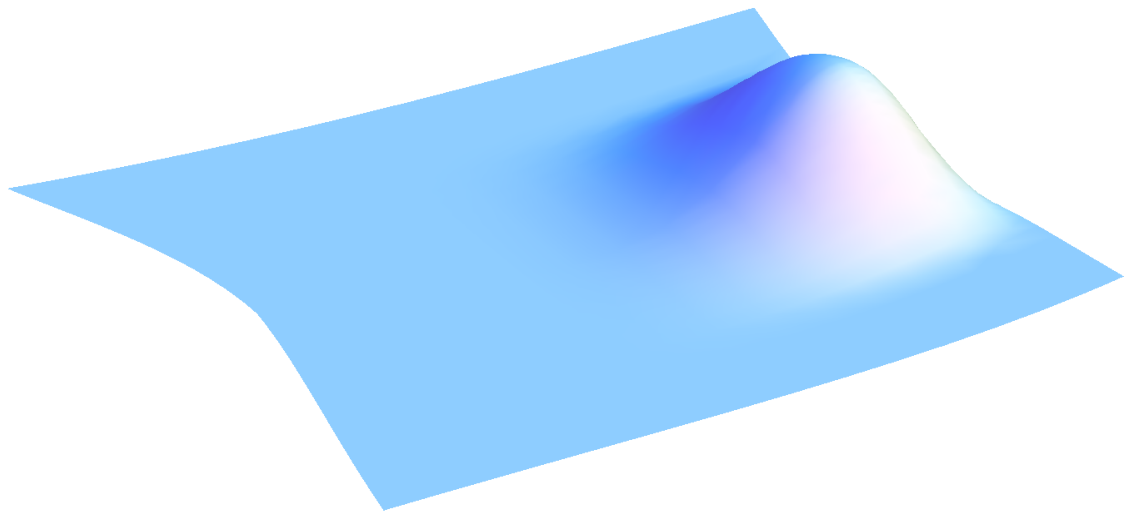} \\
$\phi_{0,2}$ & $\phi_{0,3}$ & $\phi_{0,4}$ \\
\includegraphics[width=4.1cm,clip]{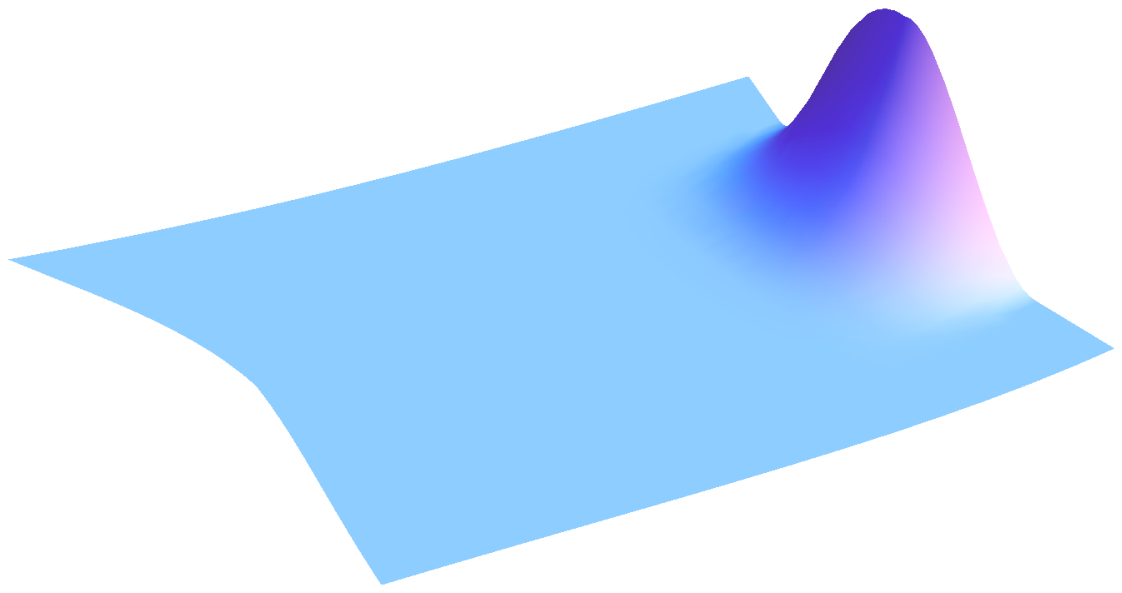}  &
\includegraphics[width=4.1cm,clip]{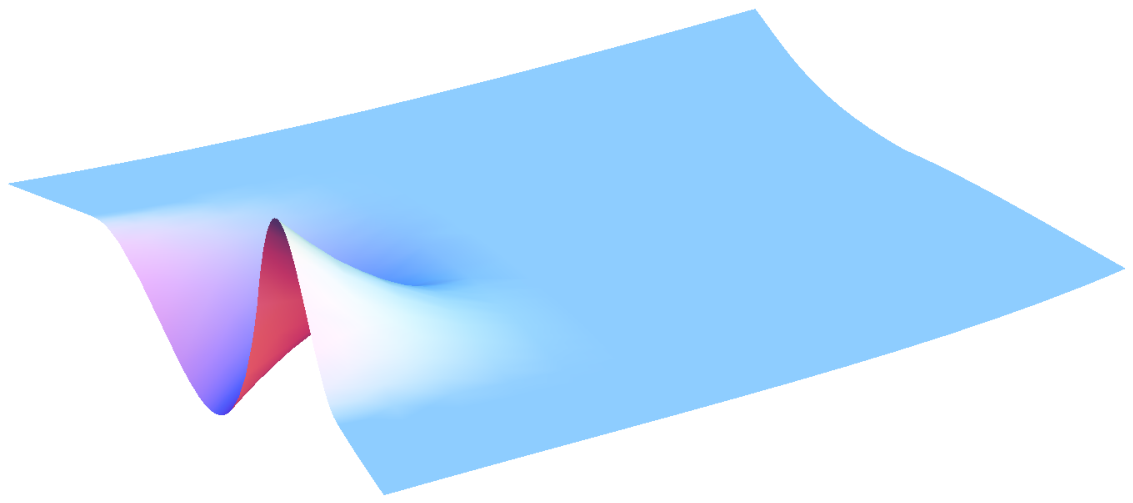} &
\includegraphics[width=4.1cm,clip]{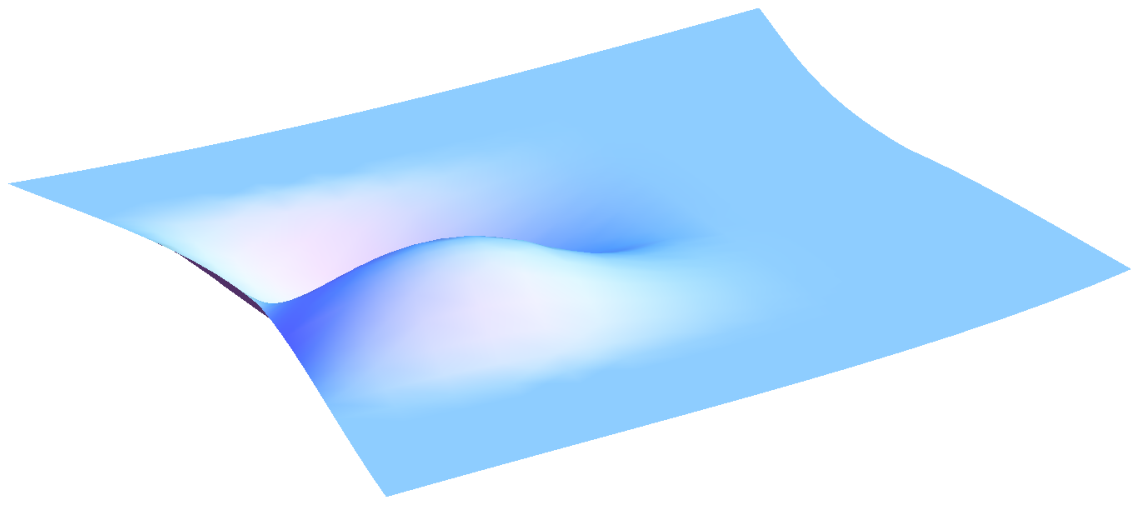} \\
$\phi_{0,5}$ & $\phi_{1,0}$ & $\phi_{1,1}$ \\
\includegraphics[width=4.1cm,clip]{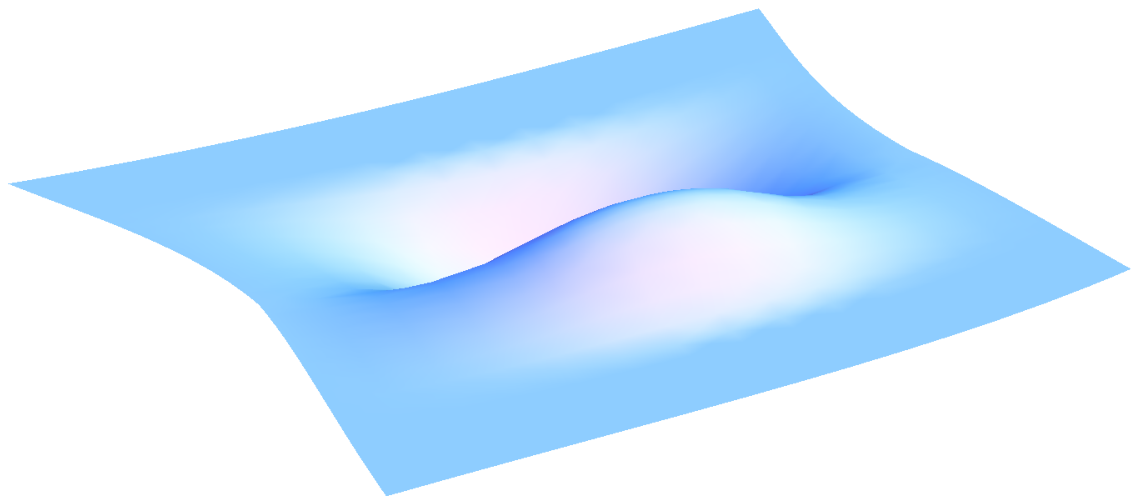}  &
\includegraphics[width=4.1cm,clip]{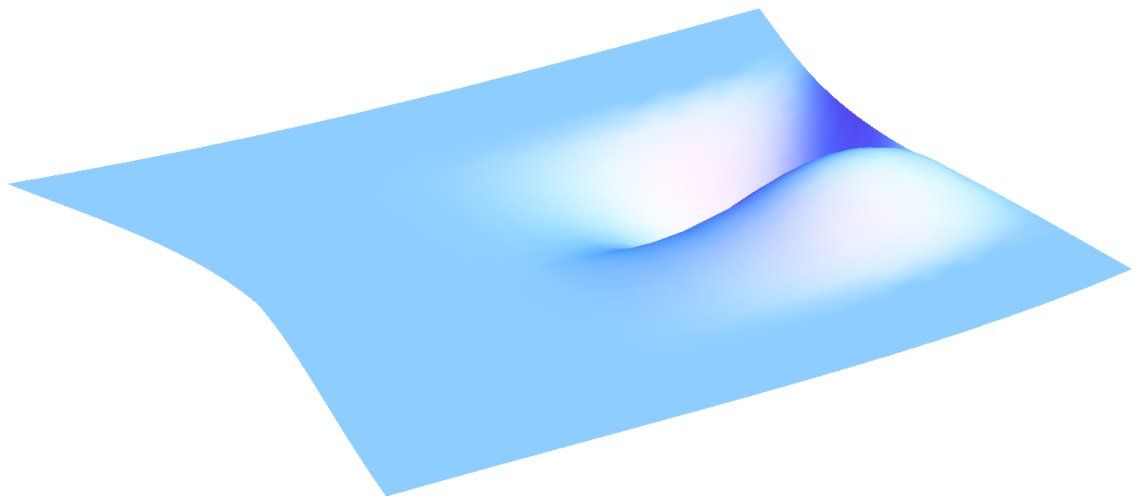} &
\includegraphics[width=4.1cm,clip]{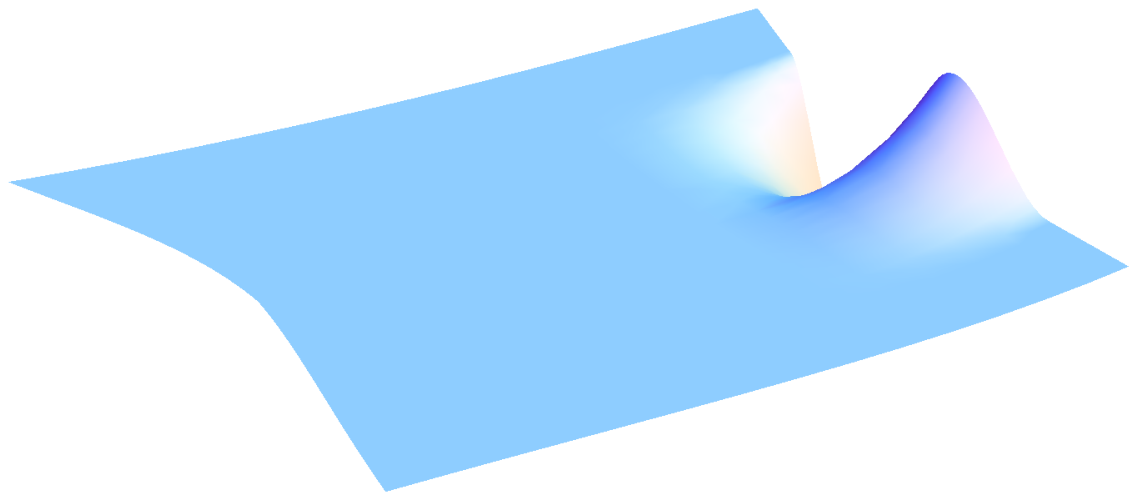} \\
$\phi_{1,2}$ & $\phi_{1,3}$ & $\phi_{1,4}$
\end{tabular}
\caption{The graphs of the basis functions $\phi_{0,i}$ and $\phi_{1,j}$ of the resulting space $\mathcal{V}^{1}_{2}$ for the given bicubic AS $G^{1}$ two-patch geometry,
when both patches $\ab{F}^{(L)}$ and $\ab{F}^{(R)}$ are represented in the space $\mathcal{S}(\mathcal{T}^{3,1}_{k},[0,1]^{2})$ for $k=2$. 
(The graphs are plotted in the parameter range $[0,\frac{1}{2}] \times [0,1]$ for both patches $\ab{F}^{(L)}$ and $\ab{F}^{(R)}$.)}
\label{fig:ex_functions}
\end{figure}

Let $\phi_{i}$, $i=0,\ldots,22 + 26 k + 8 k^2$, be the $C^{1}$-smooth isogeometric basis functions of $\mathcal{V}^{1}$, which are collected as follows:
\[
 \phi_{i}= \left\{
\begin{array}{cl}
 \phi^{(L)}_{\lfloor i/(4+2k) \rfloor+2, i \bmod (4+2k)} & \mbox{ if }0 \leq i \leq 7 + 12 k + 4 k^2 \\
 \phi^{(R)}_{\lfloor (i-8 - 12 k - 4 k^2)/(4+2k) \rfloor+2, i \bmod (4+2k)}  & \mbox{ if } 8 + 12 k + 4 k^2 \leq i \leq 15 + 24 k + 8 k^2\\
 \phi_{0,i-16 -24 k - 8 k^2} & \mbox{ if }16 + 24 k + 8 k^2 \leq i \leq 19 + 25 k + 8 k^2,\\
 \phi_{1,i-20 - 25 k - 8 k^2} & \mbox{ if }20 + 25 k + 8 k^2 \leq i \leq 22+26k + 8 k^{2}.
\end{array}
 \right.
\] 
In addition, we denote by $g^{(S)}_{i}$, $S \in \{L,R \}$, the associated spline functions $\phi_{i} \circ \ab{F}^{(S)}$. Let us consider the mass matrix 
$M=(m_{i,j})_{i,j \in \{0,\ldots, 22 + 26 k + 8 k^2 \}}$ with the entries
\[
 m_{i,j} = \int_{\Omega} \phi_{i} (\ab{x}) \phi_{j}(\ab{x}) \mathrm{d}\ab{x} = 
 \sum_{S \in \{L,R \}} \int_{0}^{1} \int_{0}^{1} g^{(S)}_{i}(u,v) g^{(S)}_{j}(u,v) |\det (J^{(S)}(u,v))| \mathrm{d}u \mathrm{d}v,
\]
where $J^{(S)}$, $S \in \{L,R \}$, is the Jacobian of $\ab{F}^{(S)}$. We also compute (for comparison) the mass matrix for the standard $C^{0}$-smooth 
isogeometric basis functions of the space $\mathcal{V}^{0}=\mathcal{V} \cap C^{0}(\Omega)$. 
Table~\ref{tab:cond} reports for different $k$ the condition numbers $\kappa$ of the diagonally scaled mass matrices $M$ (cf. \cite{Br95}) for the two different bases. 
The results indicate that the basis functions $\phi_i$ are as well conditioned as the standard $C^{0}$-smooth isogeometric basis functions.

\begin{table}
  \centering\advance\tabcolsep by -3pt 
  \begin{tabular}{|c||c|c|c|c|c|c|c|c|c|c|c|c|} \hline
 $k$  & 0    & 1 & 2    & 3    & 4 & 5    & 10    & 15 & 20    & 30    & 40  \\ \hline
 $\mathcal{V}^{0}$ & 938.91   & 724.57 & 654.29 & 624.71  & 608.64 & 598.63 & 578.41  & 572.41 & 569.93 & 567.89  & 567.31 \\ \hline
 $\mathcal{V}^{1}$ & 273.49   & 425.71 & 520.62 & 552.6   & 564.4  & 569.07 & 571.02  & 569.51 & 568.51 & 567.58  & 567.28  \\ \hline
  \end{tabular}
  \caption{For different $k$, the condition numbers $\kappa$ of the diagonally scaled mass matrices $M$ for the standard $C^{0}$-smooth 
isogeometric basis functions ($\mathcal{V}^{0}$) and our $C^{1}$-smooth isogeometric basis functions ($\mathcal{V}^{1}$), cf. Example~\ref{ex:functions}.}
  \label{tab:cond}
 \end{table}
\end{ex}

\section{Spline coefficients of the basis functions of the space $\mathcal{V}^{1}_{2}$} \label{sec:matrices}

We represent the spline functions $g^{(S)} = \phi \circ \ab{F}^{(S)}$, $S \in \{L,R \}$, for the previous constructed isogeometric basis functions 
$\phi$ of the space $\mathcal{V}^{1}_{2}$ as a linear combination of the tensor-product B-splines $N_{i,j}^{p,r}$. Thereby, these linear factors (i.e. the B-spline 
coefficients of the spline functions $g^{(S)}$ with respect to the space $\mathcal{S}(\mathcal{T}^{p,r}_{k},[0,1]^{2})$) will be described by means of blossoming. 
For the sake of simplicity (especially with respect to notation), we
will restrict ourselves below to the case ${\beta}=0$ or
$z_{\beta}=0$.  Note that our framework could be also extended to the
remaining cases. 

\subsection{Concept of blossoming}\label{subsec:blossoming}

We give a short overview of the concept of blossoming. For more detail we refer to e.g.~\cite{ChRiCo09,GO03, Li97, Ra89, Se93}. Given a 
univariate spline function $h \in \mathcal{S}(\mathcal{T}^{p,r}_{k})$ with the spline representation~\eqref{eq:spline_representation}, there exists a 
uniquely defined function $H: \R^{p} \rightarrow \R$, called the blossom of $h$, possessing the following properties:
\begin{itemize}
 \item $H$ is symmetric,
 \item $H$ is multi-affine, and
 \item $H(t,\ldots,t)=h(t)$.
\end{itemize}
These properties imply (by the so called dual function property, see \cite{GO03})  
\[
 H(t^{p,r}_{i+1},\ldots,t^{p,r}_{i+p})=d_{i},\mbox{ } i=0, \ldots , p+k(p-r),
\]
and therefore fully determine the blossom, since the value $H(s_{1},\ldots,s_{p})$ for arbitrary values $s_{1},\ldots,s_{p}$ can be 
computed by recursively using the convex combinations
\begin{align}
 & H(t^{p,r}_{i+1},\ldots,t^{p,r}_{i+p-m},s_{1},\ldots,s_{m}) = \nonumber \\
 & (1-\gamma_{i}^{m}(s_{m}))H(t^{p,r}_{i},\ldots,t^{p,r}_{i+p-m},s_{1},\ldots,s_{m-1}) + 
 \gamma_{i}^{m}(s_{m})H(t^{p,r}_{i+1},\ldots,t^{p,r}_{i+p-m+1},s_{1},\ldots,s_{m-1}), \nonumber
\end{align}
for $i=m, \ldots, p+k{p-r}$, $m=1,\ldots, p$ and 
\[
 \gamma_{i}^{m}(s) = \frac{s - t^{p,r}_{i}}{t^{p,r}_{i+p-m+1} - t^{p,r}_{i}} .
\]
The concept of blossoming provides a simple way to perform knot insertion, to differentiate a spline function and to multiply two spline functions. Given the 
spline function $h \in \mathcal{S}(\mathcal{T}^{p,r}_{k},[0,1])$ with the blossom~$H$. Representing $h$ as a spline function in the space 
$\mathcal{S}(\mathcal{T}^{p,r-1}_{k},[0,1])$, the corresponding spline control points $\bar{d}_{i}$ are given by
\[
 \bar{d}_{i} = H(t^{p,r-1}_{i+1},\ldots,t^{p,r-1}_{i+p}), \quad i=0, \ldots , p+k(p-r+1) .
\]
The spline control points $\tilde{d}_{i}$ of the derivative of $h$, i.e. $h' \in 
\mathcal{S}(\mathcal{T}^{p-1,r-1}_{k},[0,1])$, can be computed as follows:
\[
 \tilde{d}_{i} = \frac{p}{t_{i+p+1}^{p,r}-t_{i+1}^{p,r}} ( H(t_{i+1}^{p-1,r-1},\ldots, t_{i+p-1}^{p-1,r-1},t_{i+p+1}^{p,r}) - 
 H(t_{i+1}^{p-1,r-1},\ldots, t_{i+p-1}^{p-1,r-1},t_{i+1}^{p,r}))
\]
for $i=0, \ldots , p-1+k(p-r)$.
Given further the spline function $h_{1} \in \mathcal{S}(\mathcal{T}^{p_1,r_1}_{k},[0,1])$ with the blossom $H_{1}$.  Let $\hat{p}=p+p_{1}$ and $\hat{r}=\min (r,r_{1})$. 
Then the spline control points $\hat{d}_{i}$ of the product $\hat{h}=h h_{1} \in \mathcal{S}(\mathcal{T}^{\hat{p},\hat{r}}_{k},[0,1])$ are given by
\[
 \hat{d}_{i} = \frac{1}{ {\hat{p} \choose p} } \sum  H(t_{i_{1}}^{\hat{p},\hat{r}},\ldots,t_{i_{p}}^{\hat{p},\hat{r}}) H_{1}(t_{i_{p+1}}^{\hat{p},\hat{r}}, 
 \ldots , t_{i_{\hat{p}}}^{\hat{p},\hat{r}}) , \quad i=0, \ldots , \hat{p}+k(\hat{p}-\hat{r}),  
\]
where the summation runs over all possibilities to split the set $\{i+1, \ldots , i+\hat{p} \}$ into the two disjoint subsets $\{i_{1},\ldots, i_{p} \}$ 
and $\{ i_{p+1} , \ldots , i_{\hat{p}} \}$.

\subsection{Spline coefficients as matrix entries}

Let $n=p+1+k(p-r)$, $\tilde{n}=p+1+k(p-r-1)+ z_{\beta}$ and $\bar{n}=p+1-d_{\alpha}+k(p+1-d_{\alpha}-r)$. We denote by $g_{0,i}^{(S)}$,  $i=0,\ldots, \tilde{n}-1$, 
and $g_{1,j}^{(S)}$, $j=0,\ldots, \bar{n}-1$, for $S \in \{L,R \}$ the spline functions $\phi_{0,i} \circ \ab{F}^{(S)}$, and $\phi_{1,j} \circ \ab{F}^{(S)}$, 
respectively. In addition, let $\ab{B}^{(S)}= (\ab{B}^{(S)}_0, \ab{B}^{(S)}_1)^T$ be the vector of functions, where 
\[
\ab{B}^{(S)}_0(u,v) = (g_{0,0}^{(S)}(u,v),\ldots,g_{0,\tilde{n}-1}^{(S)}(u,v))^T ,
\]
as well as 
\[
\ab{B}^{(S)}_1(u,v) = (g_{1,0}^{(S)}(u,v),\ldots,g_{1,\bar{n}-1}^{(S)}(u,v))^T ,
\]
and let $\ab{B}^{*}= (\ab{B}^{*}_0, \ab{B}^{*}_1)^T$ be the vector of functions, where 
\[
\ab{B}^{*}_0(u,v)=(N_{0,0}^{p,r}(u,v),\ldots,N_{0,n-1}^{p,r}(u,v))^T,
\]
as well as 
\[
\ab{B}^{*}_1(u,v)=(N_{1,0}^{p,r}(u,v),\ldots,N_{1,n-1}^{p,r}(u,v))^T.
\]
Then there exists a matrix $A^{(S)} \in \R^{(\tilde{n}+\bar{n}) \times 2n }$ 
such that
\[
\ab{B}^{(S)}(u,v) = A^{(S)} \ab{B}^{*}(u,v). 
\]
The matrix $A^{(S)}$ for $S \in \{L,R \}$ has a block structure, resulting in an equation of the form 
\begin{equation} \label{eq:block_matrix}
\left( \begin{array}{c}
\ab{B}^{(S)}_0 \\
\ab{B}^{(S)}_1
\end{array}
\right) =
\left( \begin{array}{cc}
A_{1} & A_{2}^{(S)} \\
0 & A_{3}^{(S)}
\end{array}
\right) \left( \begin{array}{c}
\ab{B}^{*}_0 \\
\ab{B}^{*}_1
\end{array}
\right) ,
\end{equation}
where $A_{1} \in \R^{\tilde{n} \times n} $, $A_{2}^{(S)} \in  \R^{\tilde{n} \times n} $ and 
$A_{3}^{(S)} \in \R^{\bar{n} \times n} $. For each row the single entries of the matrix $A^{(S)}$, $S \in \{L,R \}$, provide the B-spline coefficients of the 
corresponding spline function $g_{0,i}^{(S)}$ or $g_{1,j}^{(S)}$ with respect to the spline space $\mathcal{S}(\mathcal{T}^{p,r}_{k},[0,1]^{2})$.

The following lemma provides the entries of the single matrices $A_{1}$, $A_{2}^{(S)}$ and $A_{3}^{(S)}$:

\begin{lem} \label{lem:matrices}
 We denote  by $H_{i}^{p,r}$ the blossom of the B-spline $N_{i}^{p,r}$. For a linear function $w: [0,1] \rightarrow \R$ with the 
 B\'ezier representation
\begin{equation} \label{eq:linearBezier}
w(t) = w_{0}(1-t) + w_{1} t , \quad t \in [0,1], \quad w_{0},w_{1} \in \R,
\end{equation}
we define the matrix $\hat{A}{(w)} = (\hat{a}^{(w)}_{i,j})_{i,j}$ given by
\[
\hat{a}^{(w)}_{i,j}= \frac{1}{p} \sum_{l=1}^{p} ( (1-t_{l}^{p,r})w_{0} + t_{l}^{p,r} w_{1} ) H_{i}^{p-1,r}(t_{j+1}^{p,r},\ldots,t_{l-1}^{p,r},t_{l+1}^{p,r},
\ldots,t_{j+p}^{p,r}).
\] 
The matrices $A_{1}$, $A_{2}^{(S)}$ and $A_{3}^{(S)}$, compare~\eqref{eq:block_matrix}, are given as follows (depending on 
$\alpha^{(S)}$, $\beta^{(S)}$ or $\beta$):
 \begin{itemize}
  \item Let $\beta= 0$. Then we have
  \[
   A_{1}=A^{(S)}_{2} = I_n , \mbox{ and }A^{(S)}_{3}= \left\{\begin{array}{cl} 
                                        \frac{1}{(k+1)p} \alpha^{(S)} I_n & \mbox{ if } d_{\alpha}=0, \\
                                        \frac{1}{(k+1)p} \hat{A}{(\alpha^{(S)})} & \mbox{ if }d_{\alpha}=1.
                                        \end{array} \right.
  \]
  \item Let $z_{\beta}=0$. Then we have 
  \[
    A_{1}=\bar{A}, \mbox{ } A^{(S)}_{2} = \bar{A} +\frac{1}{(k+1)p} \tilde{A} \hat{A}{(\beta^{(S)})} , 
    \mbox{ and }A^{(S)}_{3}= \left\{\begin{array}{cl} 
                                        \frac{1}{(k+1)p} \alpha^{(S)} I_n & \mbox{ if } d_{\alpha}=0, \\
                                        \frac{1}{(k+1)p} \hat{A}{(\alpha^{(S)})} & \mbox{ if }d_{\alpha}=1,
                                        \end{array} \right.
  \]
where the entries of the matrices $\bar{A} = (\bar{a}_{i,j})_{i,j}$ and $\tilde{A} = (\tilde{a}_{i,j})_{i,j}$ are given by
\[
\bar{a}_{i,j}=H_{i}^{p,r+1}(t_{j+1}^{p,r},\ldots,t_{j+p}^{p,r})
\]
and
\[
\tilde{a}_{i,j}= \frac{p}{t_{j+p+1}^{p,r+1}-t_{j+1}^{p,r+1}} ( H_{i}^{p,r+1}(t_{j+1}^{p-1,r},\ldots, t_{j+p-1}^{p-1,r},t_{j+p+1}^{p,r+1}) - 
 H_{i}^{p,r+1}(t_{j+1}^{p-1,r},\ldots, t_{j+p-1}^{p-1,r},t_{j+1}^{p,r+1})),
\]
respectively.
 \end{itemize}
Here, $I_n$ is the identity matrix of dimension $n$.
\end{lem}
\begin{proof}
The results follow directly from the concept of blossoming as presented in Subsection \ref{subsec:blossoming}.
\end{proof}

\begin{rem}
 The matrices $A_{1}$, $A_{2}^{(S)}$ and $A_{3}^{(S)}$, see \eqref{eq:block_matrix}, are sparse matrices (compare Example~\ref{ex:matrixp3a1}).
\end{rem}

\begin{rem}
 A further possibility to construct the matrices $A_{1}$, $A_{2}^{(S)}$ and $A_{3}^{(S)}$, see~\eqref{eq:block_matrix}, is the use of the concept of fitting. Thereby, 
 the $m$-th row of the matrices 
 $A_{1}$, $A_{2}^{(S)}$ and $A_{3}^{(S)}$, denoted by $(a_{1,m,0}, \ldots, a_{1,m,n-1})$, $(a^{(S)}_{2,m,0}, \ldots, 
 a^{(S)}_{2,m,n-1})$ and $(a^{(S)}_{3,m,0}, \ldots, a^{(S)}_{3,m,n-1})$, respectively, are computed by minimizing the terms
 \[
  \sum_{i=0}^{n-1} (g^{(L)}_{0,m}(0,\xi_{i}) - \sum_{j}^{n-1} a_{1,m,j} N_{j}^{p,r}(\xi_{i}) )^{2},
 \]
 \[
  \sum_{i=0}^{n-1} (  \frac{D_u g^{(S)}_{0,m}(0,\xi_{i})}{p} + g^{(S)}_{0,m}(0,\xi_{i}) - \sum_{j}^{n-1} a^{(S)}_{2,m,j} N_{j}^{p,r}(\xi_{i}) )^{2}
 \]
 and
 \[
  \sum_{i=0}^{n-1} (  \frac{D_u g^{(S)}_{1,m}(0,\xi_{i})}{p} + g^{(S)}_{1,m}(0,\xi_{i}) - \sum_{j}^{n-1} a^{(S)}_{3,m,j} N_{j}^{p,r}(\xi_{i}) )^{2},
 \]
 respectively, where $(\xi_{i})_{i=0,\ldots,n-1}$ are the Greville abscissa of the B-splines $N_{j}^{p,r}$, $j=0,\ldots,n-1$, of the spline space 
 $\mathcal{S}(\mathcal{T}_{k}^{p,r},[0,1])$.
\end{rem}

\begin{ex} \label{ex:matrixp3a1}
 Let ${\beta}\neq 0$ and $z_{\beta}=0$, $p=3$, $r=1$, $k \geq 2$ and $\tau_{i}=\frac{i}{k+1}$ for $i=1,\ldots,k$. Let $w: [0,1] \rightarrow \R$ be a linear function with 
 the B\'ezier representation~\eqref{eq:linearBezier}. Then the rows $\bar{a}_{i}$, $\tilde{a}_{i}$ and $\hat{a}_{i}^{(w)}$ of the matrices $\bar{A}$, 
 $\tilde{A}$ and $\hat{A}(w)$ from Lemma~\ref{lem:matrices} are given as follows:
 \begin{scriptsize}
\[
\bar{a}_{i}=\begin{cases}
(1,\underbrace{0,\ldots,0}_{2k+3}) \quad i=0\\
(0,1,\frac{1}{2},\underbrace{0,\ldots,0}_{2k+1}) \quad i=1 \\
(0,0,\frac{1}{2},\frac{2}{3},\frac{1}{3},\underbrace{0,\ldots,0}_{2k-1}) \quad i=2 \\
(\underbrace{0,\ldots,0}_{2i-3},\frac{1}{3},\frac{2}{3},\frac{2}{3},\frac{1}{3},\underbrace{0,\ldots,0}_{2(k-i)+3}) \quad 3 \leq i \leq k \\
(\underbrace{0,\ldots,0}_{2k-1},\frac{1}{3},\frac{2}{3},\frac{1}{2},0,0)\quad i=k+1 \\
(\underbrace{0,\ldots,0}_{2k+1},\frac{1}{2},1,0) \quad i=k+2 \\
(\underbrace{0,\ldots,0}_{2k+3},1) \quad i=k+3
\end{cases} \mbox{, }
\tilde{a}_{i}=\begin{cases}
(-3(k+1),\underbrace{0,\ldots,0}_{k+2}) \quad i=0 \\
(3(k+1),-\frac{3(k+1)}{2},\underbrace{0,\ldots,0}_{k+1}) \quad i=1 \\
(0,\frac{3(k+1)}{2},-(k+1),\underbrace{0,\ldots,0}_{k}) \quad i=2 \\
(\underbrace{0,\ldots,0}_{i-1},k+1,-(k+1),\underbrace{0,\ldots,0}_{k-i+2}) \quad 3 \leq i \leq k  \\
(\underbrace{0,\ldots,0}_{k},k+1,-\frac{3(k+1)}{2},0) \quad i=k+1 \\
(\underbrace{0,\ldots,0}_{k+1},\frac{3(k+1)}{2},-3(k+1))  \quad i=k+2 \\
(\underbrace{0,\ldots,0}_{k+2},3(k+1)) \quad i=k+3
\end{cases}
\]
\end{scriptsize}
and
\begin{scriptsize}
\[
\hat{a}^{(w)}_{i}= \frac{1}{6(k+1)}* \begin{cases}
(6(k + 1)w_0, 2 k w_0 + 2 w_1,\underbrace{0,\ldots,0}_{2k+2}) \quad i=0\\
(0,4(k + 1) w_0, (1 + 5 k) w_0 + 4 w_1, (k - 1) w_0 + 2 w_1,\underbrace{0,\ldots,0}_{2k}) \quad i=1\\
\left\{ \begin{array}{l}
(\underbrace{0,\ldots,0}_{2(i-1)},(k + 3 - i) w_0 + (i - 2) w_1, (9 + 5(k - i) ) w_0 + (6 + 5 (i - 2)) w_1, \\
 (6 + 5 (k - i)) w_0 + (9 + 5(i - 2) ) w_1, (k - i) w_0 + (i+1) w_1,\underbrace{0,\ldots,0}_{2(k-i+1)}) 
 \end{array} \right\} \quad 2 \leq i \leq k \\
(\underbrace{0,\ldots,0}_{2k},2 w_0 + (k - 1) w_1, 4 w_0 + (1 + 5 k) w_1, 4 (k + 1) w_1,0) \quad i=k+1 \\
(\underbrace{0,\ldots,0}_{2k+2},2 w_0 + 2 k w_1, 6 (k + 1) w_1)\quad i=k+2
\end{cases}
\]
\end{scriptsize}
\end{ex}

\section{Conclusion} \label{sec:conclusion}

We have studied the spaces of $C^{1}$-smooth isogeometric functions over a special class of two-patch geometries, so-called analysis-suitable $G^{1}$ (AS $G^{1}$) 
two-patch parameterizations (cf. \cite{CoSaTa16}). This class of two-patch geometries is of particular interest, since exactly these geometries allow under 
certain assumptions $C^{1}$ isogeometric spaces with optimal approximation properties, see \cite{CoSaTa16}. More precisely, 
we have computed the dimension of these $C^{1}$ spaces and have presented an explicit basis construction. The resulting basis functions are well conditioned, have small 
local supports and their spline coefficients can be simply computed by means of blossoming or fitting. 

Note that the constructed basis interpolates traces and transversal derivatives at the interface. Hence, the basis functions may be negative. 
In fact, the functions interpolating the transversal derivative are by construction positive on one side of the interface and negative on the other. 
The presented basis can be transformed easily to obtain locally supported basis functions which sum up to one. However, 
it is unclear whether or not a non-negative, local partition of unity exists for all AS $G^{1}$ parameterizations. 
This will be of interest for future research. 

One issue that remains to be studied is the flexibility of AS $G^1$ geometries over general multi-patch domains. 
The basis construction over two-patch geometries can be applied also to 
multi-patch configurations except for the basis functions around vertices, where modifications might be necessary. 
We are confident that the presented basis representation can be extended to the multi-patch case and used for fitting procedures, that 
approximate any given geometry with an AS $G^1$ parameterization. 

The developed basis provides a simple representation that can be implemented in existing IGA libraries. 
Thus, our $C^{1}$-smooth functions may be used to discretize different fourth-order partial differential equations.
It may be of interest for future studies to perform such simulations and analyze their properties 
as well as to investigate the class of AS $G^{1}$ parameterizations for volumetric two-patch and multi-patch domains. 

\paragraph*{\bf Acknowledgment}
The authors wish to thank the anonymous reviewers for their comments that helped to improve the paper. The first two authors (M. Kapl and G. Sangalli) were partially 
supported by the European Research Council through the FP7 ERC Consolidator Grant n.616563 HIGEOM, and by the Italian MIUR through the PRIN “Metodologie innovative 
nella modellistica differenziale numerica”. This support is gratefully acknowledged.

\end{document}